\documentclass[12pt]{article}

\usepackage{amssymb}
\usepackage{amsthm}
\usepackage{amsmath}
\usepackage{graphicx}
\usepackage{fullpage}
\usepackage{color}
\allowdisplaybreaks
\numberwithin{equation}{section}

\theoremstyle{plain}
\newtheorem{thm}{Theorem}[section]
\newtheorem{cor}[thm]{Corollary}

\newtheorem{lem}[thm]{Lemma}
\newtheorem{prop}[thm]{Proposition}

\theoremstyle{definition}
\newtheorem{defn}[thm]{Definition}

\theoremstyle{remark}

\newtheorem{rem}[thm]{Remark}

\newcommand{\N}{\mathbb{N}}
\newcommand{\M}{\mathbb{M}}
\newcommand{\R}{\mathbb{R}}

\newcommand{\Q}{\mathbb{Q}}

\newcommand{\I}{\infty}

\newcommand{\ffint}{\iint_{Q_r} \!\!\!\!\!\!\!\!\!\!\!\!\!\!\text{-----}}

\newcommand{\ffRint}{\iint_{Q_R} \!\!\!\!\!\!\!\!\!\!\!\!\!\!\text{-----}}

\newcommand{\intQrtau}{\iint_{Q_{r\tau}} \!\!\!\!\!\!\!\!\!\!\!\!\!\!\!\!\text{-----}}
\newcommand{\intQvarkr}{\iint_{Q_{\vartheta^kr}} \!\!\!\!\!\!\!\!\!\!\!\!\!\!\!\!\!\!\text{-----}}
\newcommand{\fffint}{\iint_{Q_1} \!\!\!\!\!\!\!\!\!\!\!\!\!\!\text{-----}}
\newcommand{\fthetaint}{\iint_{Q_{\vartheta}} \!\!\!\!\!\!\!\!\!\!\!\!\!\!\text{-----}}

\newcommand{\Div}{\text{div}}
\newcommand{\bp}{\begin{proof}[\ensuremath{\mathbf{Proof}}]}
\newcommand{\ep}{\end{proof}}
\newcommand{\bv}{{\bf v}}
\newcommand{\ba}{{\bf a}}
\newcommand{\bu}{{\bf u}}
\newcommand{\bg}{{\bf g}}
\newcommand{\bF}{{\bf f}}
\newcommand{\bh}{{\bf h}}
\newcommand{\bw}{{\bf w}}

\newcommand{\bzero}{{\bf 0}}
\newcommand{\Mmn}{\M^{m\times n}}
\newcommand{\Mnn}{\M^{n\times n}}

\begin{document}

\title{Compactness methods for doubly nonlinear \\ parabolic systems}

\author{Ryan Hynd\footnote{Department of Mathematics, University of Pennsylvania.  Partially supported by NSF grant DMS-1301628.}}
\maketitle

\begin{abstract}
We study solutions of the system of PDE $D\psi(\bv_t)=\Div DF(D\bv)$, where $\psi$ and $F$ are convex functions. This type of system arises 
in various physical models for phase transitions.  We establish compactness properties of solutions that allow us to verify partial regularity when $F$ is quadratic and 
characterize the large time limits of weak solutions.  Special consideration is also given to systems that are homogeneous and their connections with nonlinear eigenvalue problems. 
While the uniqueness of weak solutions of such systems of PDE remains an open problem, we show scalar equations always have a preferred solution 
that is also unique as a viscosity solution.
\end{abstract}



\section{Introduction}
 In this paper, we consider solutions $\bv: \Omega\times(0,T)\rightarrow \R^m$ of the system of PDE
\begin{equation}\label{mainPDE}
D\psi(\bv_t)=\Div DF(D\bv)
\end{equation} 
where $\Omega\subset\R^n$ is a bounded domain with smooth boundary and $T>0$. Here $\psi\in C^1(\R^m)$ and $F\in C^1(\M^{m\times n})$ are strictly convex functions, $\M^{m\times n}$ is the space of $m\times n$ matrices with real entries, and $\bv=(v^1,\dots, v^m)$ has $m$ component functions $v^i=v^i(x,t)$. Also note $\bv_t=(v^i_t)\in \R^m$ is the time derivative and 
$D\bv=(v^i_{x_j})\in \M^{m\times n}$ is the spatial gradient matrix of $\bv$.  Writing $\psi(w)=\psi(w^1,\dots, w^m)$ for $w=(w^i)\in \R^m$ and $F(M)=F(M^1_1, \dots, M^m_n)$ for $M=(M^i_j)\in \Mmn$, the system \eqref{mainPDE} can also be posed as the $m$ equations
$$
\psi_{w_i}(\bv_t)=\sum^n_{j=1}\left(F_{M^i_j}(D\bv)\right)_{x_j}, \quad i=1, \dots, m. 
$$

\par Equation \eqref{mainPDE} is known as a {\it doubly nonlinear evolution} and arises in mathematical models for phase transitions in thermodynamics, ferro-magnetism, plasticity theory, 
and other phenomena involving dissipation \cite{Colli, Mielke, Visintin1, Visintin}.  Much of the analysis of the resulting model equations concern abstract generalizations of \eqref{mainPDE}. In this paper, we 
focus entirely on the system \eqref{mainPDE} and use PDE techniques to derive refined results on solutions.   

\par  In the work that follows, we present three theorems related to the theory of doubly nonlinear evolutions. The first involves the existence of a particular weak solution of the initial value problem
\begin{align}\label{mainIVP}
\begin{cases}
D\psi(\bv_t)=\Div DF(D\bv), &\quad \Omega\times(0,T) \\
\hspace{.42in}\bv =\bzero, &\quad \partial\Omega\times [0,T)  \\ 
\hspace{.42in}\bv =\bg, &\quad \Omega\times\{0\}
\end{cases}.
\end{align} 
Here $\bg:\Omega\rightarrow \R^m$ is a given function and $\bzero:\R^n\rightarrow \R^m$ is the mapping that is identically equal to $0\in \R^m$. The existence of a weak solution as defined in 
Definition \ref{WeakDefn} below has been established for evolutions related to \eqref{mainPDE}; see for instance the seminal works \cite{Arai, Colli2,Colli} and Proposition \ref{ExistProp} of this paper. 

\par However, the uniqueness of solutions remains largely open. Of course when $D\psi(w)=w$, \eqref{mainPDE}  is naturally interpreted as an $L^2(\Omega; \R^m)$ gradient flow 
associated with the convex functional $\bu\mapsto \int_\Omega F(D{\bf u})dx$.  Initial value problems related to this system are known to be well-posed and to generate a nonlinear contraction semigroup in 
$L^2(\Omega; \R^m)$ \cite{AGS,Barbu, Brezis, Evans}.  Likewise, when $DF(M)=M$ 
\begin{equation}\label{PDEex}
D\psi(\bv_t)=\Delta \bv
\end{equation} 
has a contraction property. Indeed, if $\bv^1$ and $\bv^2$ are two smooth solutions of \eqref{PDEex} that vanish on $\partial \Omega$
\begin{align*}
\frac{d}{dt}\frac{1}{2}\int_\Omega|D\bv^1-D\bv^2|^2dx &=\int_\Omega(D\bv^1-D\bv^2)\cdot (D\bv^1_t-D\bv^2_t) dx \\
&= -\int_\Omega(\Delta \bv^1-\Delta\bv^2)\cdot (\bv^1_t-\bv^2_t) dx \\
&=- \int_\Omega(D\psi(\bv^1_t)-D\psi(\bv^2_t))\cdot (\bv^1_t-\bv^2_t) dx \\
&\le 0,
\end{align*}
as $\psi$ is convex.  Here $M\cdot N:=\text{tr}M^t N=\sum^m_{i=1}\sum^n_{j=1}M^i_j N^i_j$ and $|M|:=\sqrt{M\cdot M}$.  As a result, when either $D\psi$ or $DF$ is {\it linear}, we expect \eqref{mainPDE} to have a good existence and uniqueness theory.  It is unknown if similar is true when $D\psi$ and $DF$ are both nonlinear.


\par More can be said about the uniqueness of solutions of the initial value problem \eqref{mainIVP} when $m=1$.  In this case, the initial value problem \eqref{mainIVP} reduces to 
\begin{align}\label{scalarIVP}
\begin{cases}
\psi'(v_t)=\Div DF(Dv), &\quad \Omega\times(0,T) \\
\hspace{.33in}v =0, &\quad \partial\Omega\times [0,T)  \\ 
\hspace{.33in}v =g, &\quad \Omega\times\{0\}
\end{cases}
\end{align} 
for a scalar function $v:\Omega\times(0,T)\rightarrow \R$. Closely related to this problem is the implicit time scheme: $v^0=g$ 
\begin{equation}\label{scalarIFT}
\begin{cases}
\psi'\left(\frac{v^k - v^{k-1}}{\tau}\right)=\Div DF(Dv^k), \quad & x\in \Omega\\
\hspace{.72in}v^k=0,\quad & x\in \partial \Omega
\end{cases}
\end{equation}
for  $ k\in \N$ and $\tau>0$.   By making standard coercivity and growth assumptions such as \eqref{Coercive} and \eqref{GrowthCond} below, direct methods from the calculus of variations imply \eqref{scalarIFT} has a unique solution sequence $\{v^k_\tau\}_{k\in \N}$. We show
this sequence can be use to construct the unique viscosity solution of \eqref{scalarIVP} which is also a weak solution of \eqref{scalarIVP}. Consequently, there is always a weak solution of \eqref{scalarIVP} that can be singled out as also being a viscosity solution. 
 \begin{thm}\label{ViscSolnResult}
Assume $m=1$, $p\in [2,\infty)$ and $g\in C(\overline{\Omega})\cap  W^{1,p}_0(\Omega)$. Additionally suppose $F$ and $\psi$ satisfy \eqref{Coercive}, \eqref{GrowthCond}, \eqref{FCtwo} and \eqref{RegAssump} below. Denote the solution sequence of the implicit scheme \eqref{scalarIFT} as
$\{v^k_\tau\}_{k\in \N}$, and for $N\in\N$ and $T>0$ define 
\begin{equation}\label{CLaan}
v_N(\cdot, t):=
\begin{cases}
g, \hspace{.4in} t=0\\
v^k_{T/N}, \quad (k-1)T/N< t \le kT/N, \quad k=1, \dots, N\\
\end{cases}.
\end{equation}
Then $v(\cdot,t):=\lim_{N\rightarrow \infty}v_N(\cdot,t)$ exists in  $L^p(\Omega)\cap C(\overline{\Omega})$ uniformly in $t\in [0,T]$. Moreover, $v$ is the unique viscosity solution and a weak solution of the initial value problem \eqref{scalarIVP}.
\end{thm}
\noindent Theorem \ref{ViscSolnResult} is reminiscent of the Crandall-Liggett generation theorem \cite{CL}. However, the type of 
convergence detailed in our theorem is not immediate since it is unlikely that doubly nonlinear evolutions generate contraction semigroups.

\par Our second main result involves the use of compactness methods to verify partial regularity of solutions of \eqref{mainPDE} when $F$ is quadratic.  While \eqref{PDEex} is merely a special case of doubly nonlinear 
evolutions, it is a system for which little regularity theory directly applies. In fact, writing \eqref{PDEex} in terms of the Legendre transform of $\psi$ 
$$
\bv_t=D\psi^*(\Delta \bv)
$$
reveals that \eqref{PDEex} is an example of a fully nonlinear parabolic {\it system}. In particular, when $m>1$ viscosity solutions methods do not apply and we must 
devise other approaches. Our main tool is a compactness assertion of the second derivatives of weak solutions of \eqref{PDEex}; see Corollary \ref{H2Compactness} below. 

\begin{thm}\label{PartialRegThm}
Assume $\bv$ is a weak solution of \eqref{PDEex} on $\Omega\times(0,T)$, $\psi\in C^2(\R^m)$ and 
$$
\alpha |z|^2\le D^2\psi(w)z\cdot z\le A|z|^2, \quad w,z\in \R^m
$$
for some $A,\alpha>0$.  If
\begin{equation}\label{DVTbound}
D\bv_t\in L^2_{\text{loc}}(\Omega\times(0,T); \Mmn),
\end{equation}
then there is an open set $S\subset \Omega\times(0,T)$ with full Lebesgue measure for which $\bv_t, D^2\bv$ are H\"{o}lder continuous at each point in $S$.  
\end{thm} 
\noindent Unfortunately, we do not know if {\it every} weak solution possesses the integrability \eqref{DVTbound}. However, we can show that there are weak solutions of the initial problem \eqref{mainIVP} that do satisfy condition \eqref{DVTbound}. See Proposition \ref{DVTProp}. 
\par Finally, we will use compactness methods to study large time limits for boundary value problems related to \eqref{mainPDE}. In Proposition 
\ref{LargeTimeMan} below, we show that the large time behavior of solutions of \eqref{mainPDE} is governed by the appropriate stationary system associated to \eqref{mainPDE}.  
What we found to be even more interesting is to consider the large time behavior of the homogeneous flow
\begin{align}\label{PflowIVP}
\begin{cases}
|\bv_t|^{p-2}\bv_t=\Div(|D\bv|^{p-2}D\bv), &\quad \Omega\times(0,\infty) \\
\hspace{.51in}\bv = \bzero, &\quad \partial\Omega\times [0,\infty)  \\ 
\hspace{.51in}\bv =\bg, &\quad \Omega\times\{0\}
\end{cases}.
\end{align} 
The system \eqref{PflowIVP} corresponds to \eqref{mainIVP} when $\psi(w)=\frac{1}{p}|w|^p$ and $F(M)=\frac{1}{p}|M|^p$ ($1<p<\infty)$ and is closely related to the minimization problem 
for the optimal $p$-Rayleigh quotient
\begin{equation}\label{Lambdap}
\Lambda_p:=\inf\left\{\frac{\int_\Omega|D\bu(x)|^pdx}{\int_\Omega|\bu(x)|^pdx}:\bu \in W^{1,p}_0(\Omega;\R^m)\setminus\{\bzero\}\right\}.
\end{equation}
Minimizing $\bu$ in \eqref{Lambdap} are called {\it $p$-ground states}.  

\par In section \ref{LongSEct} below, we verify that 
weak solutions of \eqref{PflowIVP} satisfy 
$$
\int_\Omega |D\bv(x,t)|^pdx\le e^{-p\Upsilon_p t}\int_\Omega |D\bg(x)|^pdx
$$
for $t\ge 0$ where 
\begin{equation}\label{UpsAndDowns}
\Upsilon_p:=\Lambda_p^{\frac{1}{p-1}}. 
\end{equation}
We also justify the following monotonicity of the $p$-Rayleigh quotient under the flow \eqref{PflowIVP}
$$
\frac{d}{dt}\frac{\int_\Omega |D\bv(x,t)|^pdx}{\int_\Omega |\bv(x,t)|^pdx}\le 0.
$$
These are crucial steps in the proof of the following result. 
\begin{thm}\label{HomogenousConvergence} Assume $\bg \in W^{1,p}_0(\Omega;\R^m)$ and suppose $\bv$ is a weak solution of \eqref{PflowIVP} that satisfies
$$
\liminf_{t\rightarrow \infty}e^{p\Upsilon_p t}\int_\Omega |D\bv(x,t)|^pdx>0.
$$
Then $\bv(\cdot, t)\neq \bzero$ for each $t\ge 0$ and 
\begin{equation}\label{LamLimit}
\Lambda_p=\lim_{t\rightarrow\infty} \frac{\int_{\Omega}|D\bv(x,t)|^pdx }{\int_{\Omega}|\bv(x,t)|^pdx }.
\end{equation}
Moreover, for any increasing sequence of positive numbers $\{t_k\}_{k\in \N}$ tending to $\infty$,
$\{e^{\Upsilon_p t_k}\bv(\cdot, t_k)\}_{k\in \N}$
has a subsequence that converges to a $p$-ground state in $W^{1,p}_0(\Omega,\R^m)$.
\end{thm}
\noindent In previous joint work, we showed that when $m=1$ the full limit $\lim_{t\rightarrow \infty}e^{\Upsilon_p t}v(\cdot, t)$ exists in $W^{1,p}_0(\Omega)$ and is a $p$-ground state provided it is not identically $0$ \cite{HyndLindgren}.  

\par The organization of this paper is as follows. In section \ref{CompSect}, we define weak solutions and verify a general compactness assertion of weak solutions. 
Section \ref{ITS} is a discussion of the existence of weak solutions of \eqref{mainIVP} and is where we establish Theorem \ref{ViscSolnResult}. We discuss the regularity of 
solutions in section \ref{RegSect} and also prove Theorem \ref{PartialRegThm}. Finally, in section \ref{LongSEct} we study the large time limits of solutions of \eqref{mainPDE}
and verify Theorem \ref{HomogenousConvergence}. We thank Nader Masmoudi for pointing out a refinement in our original compactness assertion (Theorem \ref{BlowLem})
and Erik Lindgren and Ovidiu Savin for interesting discussions related to this work.


\section{Compactness}\label{CompSect}
An important identity for smooth solutions of \eqref{mainIVP} is 
\begin{equation}\label{EnergyIdentity}
\frac{d}{dt}\int_{\Omega}F(D\bv(x,t))dx=-\int_\Omega D\psi(\bv_t(x,t))\cdot \bv_t(x,t)dx.
\end{equation}
This identity can of course can be integrated in time to yield 
\begin{equation}\label{EnergyIdentity2}
\int^t_0\int_\Omega D\psi(\bv_t(x,s))\cdot \bv_t(x,s)dxds  + \int_{\Omega}F(D\bv(x,t))dx =\int_{\Omega}F(D\bg(x))dx
\end{equation}
for $t\in [0,T]$.  It is now appropriate to assume there is $p\in (1,\infty)$ such that 
\begin{equation}\label{Coercive}
\psi(w)+F(M)\ge \gamma(|w|^p+|M|^p) - \beta, \quad (w,M)\in \R^m\times \Mmn
\end{equation}
for some constants $\gamma> 0$ and $\beta\ge 0$. This is typically called a {\it coercivity condition}.  As equation \eqref{mainPDE} only involves 
$D\psi$ and $DF$, we may assume without any loss of generality that both $\psi$ are $F$ are nonnegative.  We also suppose there is 
a constant $C$ such that
\begin{equation}\label{GrowthCond}
|D\psi(w)|+|DF(M)|\le C(|w|^{p-1}+|M|^{p-1} +1), \quad (w,M)\in \R^m\times\Mmn.
\end{equation}
\par The conditions \eqref{EnergyIdentity2} and \eqref{Coercive}, combined with the convexity of $\psi$, imply a bound on the quantity 
$$
\int^T_0\int_\Omega |\bv_t(x,t)|^pdxdt+\max_{0\le t\le T}\int_\Omega |D\bv(x,t)|^pdx
$$
depending on $\int_{\Omega}F(D\bg(x))dx$ and the constants in \eqref{Coercive} and \eqref{GrowthCond}.  This observation motivates the following definition of a solution of the initial value problem \eqref{mainIVP}. 
\begin{defn}\label{WeakDefn}
Assume $\bg\in W^{1,p}_0(\Omega; \R^m)$. A mapping $\bv$ satisfying 
\begin{equation}\label{NaturalSpace}
\bv_t\in L^p(\Omega\times(0,T); \R^m), \quad \bv\in L^\infty([0,T]; W^{1,p}_0(\Omega;\R^m))
\end{equation}
is a {\it weak solution} of \eqref{mainIVP} provided 
\begin{equation}\label{WeakSolnCond}
\int_\Omega D\psi(\bv_t(x,t))\cdot  \bw(x)dx + \int_\Omega DF(D\bv(x,t))\cdot D\bw(x)dx=0, 
\end{equation}
for each $\bw\in W^{1,p}_0(\Omega; \R^m)$ and Lebesgue almost every $t\in [0,T]$, and 
\begin{equation}\label{InitialCond}
\bv(x,0)=\bg(x).
\end{equation}
\end{defn} 
In view of \eqref{NaturalSpace}, any weak solution $\bv$ takes values in $L^p(\Omega; \R^m)$ continuously in time. As a result, it makes sense to require \eqref{InitialCond}.  
We will also show that a version of the identity \eqref{EnergyIdentity} holds for all weak solutions.  This task would be easy if we could choose $\bw=\bv_t(\cdot, t)$ in \eqref{WeakSolnCond}. However, we do not know if 
$D\bv_t(\cdot, t)\in L^{p}(\Omega;\Mmn)$.  To get around this difficulty, we appeal to an abstract result about convex functionals on reflexive Banach Spaces. 

\begin{lem}\label{EnergyLemma}
Assume $\bv$ is a weak solution of \eqref{mainIVP}. Then \eqref{EnergyIdentity} holds for Lebesgue almost every $t\in [0,T]$. In particular, $[0,T]\ni t\mapsto \int_\Omega F(D\bv(x,t))dx$ is absolutely continuous and 
\eqref{EnergyIdentity2} is valid for every $t\in [0,T]$. 
\end{lem}

\begin{proof}
For $\bu \in L^p(\Omega; \R^m)$ define 
\begin{equation}\label{convFuncPhi}
\Phi(\bu):=
\begin{cases}
\int_{\Omega}F(D\bu(x))dx, & \bu\in W^{1,p}_0(\Omega; \R^m)\\
+\infty, & \text{otherwise}
\end{cases}.
\end{equation}
Observe that $\Phi$ is proper, lower-semicontinuous, and convex. Moreover, \eqref{WeakSolnCond} implies 
$$
\partial \Phi(\bv(\cdot, t)) = \{-D\psi(\bv_t(\cdot,t))\}
$$
for Lebesgue almost every $t\in [0,T]$. The bound \eqref{GrowthCond} implies $t\mapsto \int_\Omega D\psi(\bv_t(x, t))\cdot \bv_t(x, t) dx \in L^1(0,T)$. By Proposition 1.4.4 in \cite{AGS}, $[0,T]\ni t\mapsto \Phi(\bv(\cdot, t))$ is absolutely continuous and the chain rule gives that \eqref{EnergyIdentity} holds for almost
every $t\in [0,T]$. Finally, \eqref{EnergyIdentity2} follows from integration. 
\end{proof}
\begin{cor}\label{EntropyLem}
Assume $\bv$ is a weak solution of \eqref{mainIVP}. If 
\begin{equation}\label{minPSI}
\min_{w\in\R^m}\psi(w)=\psi(0),
\end{equation}
then $[0,T]\ni t\mapsto \int_\Omega F(D\bv(x,t))dx$ is nonincreasing. 
\end{cor}
\begin{proof}
 As $\psi$ is convex and continuously differentiable, $0\le \psi(w)-\psi(0)\le  D\psi(w)\cdot w$ for each $w\in \R^m$. Consequently, the claim follows from
\eqref{EnergyIdentity}. 
\end{proof}
\begin{rem}
Interpreting $t\mapsto \int_\Omega F(D\bv(x,t))dx$ as a type of energy, Corollary \ref{EntropyLem} asserts that dissipation occurs when \eqref{minPSI} holds.  In particular, 
$$
\frac{d}{dt}\int_{\Omega}F(D\bv(x,t))dx\le -\int_\Omega(\psi(\bv_t(x,t))-\psi(0))dx
$$
and so $\psi$ controls the rate of dissipation. 
\end{rem}
While \eqref{mainIVP} may not have a unique solution, the only solution of the initial value problem when $\bg=\bzero$ is also equal to $\bzero$. Likewise, if $\bg$ is a minimizer of 
$\Phi$ defined in \eqref{convFuncPhi}, then any weak solution coincides with $\bg$ for all later times. 
\begin{cor}
Assume $\bv$ is a weak solution of \eqref{mainIVP}  and that $\psi$ satisfies \eqref{minPSI}. (i) If $\bg =\bzero$, then $\bv(\cdot, t)=\bzero$ for $t\in [0,T]$. 
(ii) If $\bg$ is a minimizer of $\Phi$, then $\bv(\cdot, t)=\bg$ for $t\in [0,T]$.  
\end{cor} 
\begin{proof}
$(i)$ Let $O_{m\times n}$ denote the $m\times n$ matrix with each entry equal to $0$. By our assumptions and the previous claim, $\int_\Omega F(D\bg(x))dx = |\Omega| F(O_{m\times n})\ge \int_\Omega F(D\bv(x,t))dx$ for $t\in [0,T]$. By
Jensen's inequality 
\begin{align*}
\int_\Omega F(D\bv(x,t))dx \ge |\Omega| F\left(\frac{1}{|\Omega|}\int_\Omega D\bv(x,t)dx \right)= |\Omega| F(O_{m\times n})
\end{align*}
as $\bv(\cdot, t)$ vanishes on $\partial \Omega$.  Thus, $\int_\Omega F(D\bg(x))dx = \int_\Omega F(D\bv(x,t))dx$, and it follows that from the energy identity \eqref{EnergyIdentity2} and 
 assumption \eqref{minPSI} that $\bv_t=\bzero$. 
\par $(ii)$ By  \eqref{minPSI}, $\int_\Omega F(D\bg(x))dx \ge \int_\Omega F(D\bv(x,t))dx$. Thus if $\bg$ is a minimizer of $\Phi$, so is $\bv(\cdot, t)$ for each $t\in [0,T]$. Since $F$ is strictly convex, $\Phi$ 
has a single minimizer and thus the claim follows. 
\end{proof}
The energy identity \eqref{EnergyIdentity} also can we obtained by simply multiplying the PDE \eqref{mainPDE} by $\bv_t$ and then integrating by parts. Multiplying by $\bv_{tt}$ formally gives us a non-decreasing quantity 
involving the Legendre transform of $\psi$ 
$$
\psi^*(\xi)=\sup_{w\in \R^m}\left\{w\cdot \xi -\psi(w)\right\}, \quad \xi\in \R^m.
$$ 
Surprisingly, the following assertion seems to have gone unnoticed in the literature of doubly nonlinear equations.  We also will present a rigorous version of this assertion below; see 
inequality \eqref{RigorNewDecrease}. 

\begin{prop}\label{NewDecrease}
Assume $F\in C^2(\Mmn)$ and $\bv$ is a smooth solution of \eqref{mainIVP}. Then 
$$
[0,T]\ni t\mapsto \int_\Omega\psi^*(D\psi(\bv_t(x,t)))dx
$$
is nonincreasing. 
\end{prop}
\begin{proof}
First note that
\begin{align*}
\int_\Omega\psi^*(D\psi(\bv_t(x,t)))dx & =\int_\Omega D\psi(\bv_t(x,t))\cdot \bv_t(x,t)- \psi(\bv_t(x,t))dx \\
& = - \int_\Omega(DF(D\bv(x,t))\cdot D\bv_t(x,t) +\psi(\bv_t(x,t)))dx.
\end{align*}
As a result
\begin{align*}
\frac{d}{dt}\int_\Omega\psi^*(D\psi(\bv_t(x,t)))dx &= - \int_\Omega \sum^m_{i,k=1}\sum^n_{j,l=1}F_{M^i_jM^k_l}(D\bv(x,t))(\bv_t)^i_{x_j}(x,t)(\bv_t)^k_{x_l}(x,t)dx \\
&\quad - \int_\Omega DF(D\bv(x,t))\cdot D\bv_{tt}(x,t)dx + \int_\Omega D\psi(\bv_t(x,t))\cdot \bv_{tt}(x,t)dx\\
&\le  - \int_\Omega DF(D\bv(x,t))\cdot D\bv_{tt}(x,t)dx + \int_\Omega D\psi(\bv_t(x,t))\cdot \bv_{tt}(x,t)dx\\
&=0. 
\end{align*}
\end{proof}
Now we verify an important compactness property of solutions of the initial value problem \eqref{mainIVP}.  The following claim and its proof will help 
us establish each of the main results of this paper.  

\begin{thm}\label{BlowLem}
Assume $\{\bv^k\}_{k\in \N}$ is a sequence of weak solutions of \eqref{mainIVP} and \\ $\{\bg^k:=\bv^k(\cdot,0)\}_{k\in \N}
\subset W^{1,p}_0(\Omega; \R^m)$ is a bounded sequence.  Then 
there is a subsequence $\{ \bv^{k_j}\}_{j\in \N}$ and $\bv$ satisfying \eqref{NaturalSpace} such that 
\begin{equation}\label{weakConv1}
\bv^{k_j}\rightarrow \bv \; \text{in}\; 
\begin{cases}
C([0,T],L^p(\Omega; \R^m))\\
L^r([0,T]; W^{1,p}_0(\Omega; \R^m)),\;\; 1\le r<\infty
\end{cases}
\end{equation}
and 
\begin{equation}\label{StrongConv}
\bv_t^{k_j}\rightarrow \bv_t \; \text{in}\; L^p(\Omega\times [0,T]; \R^m).
\end{equation}
Moreover, $\bv$ is a weak solution of \eqref{mainIVP} where $\bg$ is a weak limit of $\{\bg^{k_j}\}_{j\in \N}$. 
\end{thm}

\begin{proof}
1. The assumption that $\bg^k$ is bounded in $W^{1,p}_0(\Omega;\R^m)$ gives
\begin{equation}\label{UnifEnergyBound}
\sup_{k\in\N}\left\{\int^T_0\int_\Omega |\bv^k_t(x,t)|^pdxdt+\max_{0\le t\le T}\int_\Omega |D\bv^k(x,t)|^pdx
\right\}<\infty.
\end{equation}
It follows that $\{\bv^k\}_{k\in \N}\subset C([0,T]; L^p(\Omega; \R^m))$ is equicontinuous and pointwise bounded in \\ 
$W^{1,p}_0(\Omega;\R^m)$.  An abstract version of the Arzel\`{a}-Ascoli theorem, as detailed by J. Simon \cite{Simon}, implies that there is a subsequence $\{\bv^{k_j}\}_{k\in \N}$
converging to some $\bv \in C([0,T]; L^p(\Omega; \R^m))$.  Moreover, $\bv^{k_j}(\cdot, t)$ converges to 
$\bv(\cdot,t)$ weakly in $W^{1,p}_0(\Omega;\R^m)$ for each $t\in [0,T]$ and we may assume $\bv_t^{k_j}\rightharpoonup \bv_t$ in 
$L^p(\Omega\times [0,T]; \R^m).$ By \eqref{GrowthCond}, $D\psi(\bv^k_t)$ is bounded in $L^q(\Omega\times [0,T]; \R^m);$ so we may also assume that $D\psi(\bv^{k_j}_t)\rightharpoonup \Theta$. 

\par 2.  Notice that for each interval $[t_1, t_2]\subset [0,T]$ and $\bw \in W^{1,p}_0(\Omega;\R^m)$
\begin{align}\label{ConvexMakesStrong}
\int^{t_2}_{t_1}\int_\Omega F(D\bw(x))dxdt &\ge \int^{t_2}_{t_1}\int_\Omega F(D\bv^{k_j}(x,t))dxdt   \nonumber \\
&\hspace{1in}+ \int^{t_2}_{t_1}\int_\Omega DF(D\bv^{k_j}(x,t))\cdot \left(D\bw(x)-D\bv^{k_j}(x,t)\right)dxdt \nonumber \\
& = \int^{t_2}_{t_1}\int_\Omega F(D\bv^{k_j}(x,t))dxdt \nonumber \\
&\hspace{1in} - \int^{t_2}_{t_1}\int_\Omega D\psi(\bv^{k_j}_t(x,t))\cdot (\bw(x)-\bv^{k_j}(x,t))dxdt.
\end{align}
Sending $j\rightarrow\infty$ gives 
\begin{equation}\label{ThetaEq}
\Theta=\Div DF(D\bv). \nonumber
\end{equation}
In particular, we can adapt the proof of Lemma \ref{EnergyLemma} to conclude
\begin{equation}\label{ConvEq1}
\int^t_s\int_\Omega \Theta(x,\tau)\cdot\bv_t(x,\tau)dxd\tau + \int_\Omega F(D\bv(x,t))dx = \int_\Omega F(D\bv(x,s))dx 
\end{equation} 
for each $0\le s\le t\le T$.
\par 3. Choosing $\bw=\bv(\cdot, t)$ in \eqref{ConvexMakesStrong} and sending $j\rightarrow \infty$ leads to
$$
\int^{t_2}_{t_1}\int_\Omega F(D\bv(x,t))dxdt\ge\limsup_{j\rightarrow\infty}\int^{t_2}_{t_1}\int_\Omega F(D\bv^{k_j}(x,t))dxdt.
$$
Also note that since $F$ is convex, Fatou's lemma and weak convergence imply 
$$
\int^{t_2}_{t_1}\int_\Omega F(D\bv(x,t))dxdt\le\liminf_{j\rightarrow\infty}\int^{t_2}_{t_1}\int_\Omega F(D\bv^{k_j}(x,t))dxdt. 
$$
In view of the growth estimate \eqref{GrowthCond} and the strict convexity of $F$, we conclude that (up to a subsequence) $D\bv^{k_j}\rightarrow D\bv$ in $L^p(\Omega\times[0,T]; \R^m)$. 
Moreover, the uniform energy bound \eqref{UnifEnergyBound} implies for each $r>p$ there is a constant $C$ independent of $k_j$ for which
$$
\int^T_0|\bv^{k_j}(\cdot,t)-\bv(\cdot,t)|^{r}_{W^{1,p}_{0}(\Omega;\R^m)}dt\le 
C\int^T_0|\bv^{k_j}(\cdot,t)-\bv(\cdot,t)|^p_{W^{1,p}_{0}(\Omega;\R^m)}dt.
$$
Hence, $\bv^{k_j}\rightarrow \bv $ in $ L^r([0,T]; W^{1,p}_0(\Omega; \R^m))$.

\par 4. It follows that there is a subsequence of $\bv^{k_j}(\cdot, t)$ (not relabeled) converging to $\bv(\cdot, t)$ in $W^{1,p}_0(\Omega;\R^m)$ for Lebesgue almost every $t\in [0,T]$.
Let $0\le t_0\le t_1\le T$ be two such times. From Lemma \ref{EnergyLemma}, 
\begin{equation}\label{EnergyK}
\int^{t_1}_{t_0}\int_\Omega(\psi^*(D\psi(\bv^{k_j}_t)) + \psi(\bv^{k_j}_t))dxdt + \int_\Omega F(D\bv^{k_j}(x,t_1))dx= \int_\Omega F(D\bv^{k_j}(x,t_0))dx.\nonumber
\end{equation}
Letting $j\rightarrow\infty$ gives
\begin{equation}\label{ConvEq2}
\int^{t_1}_{t_0}\int_\Omega(\psi^*(\Theta) + \psi(\bv_t))dxds + \int_\Omega F(D\bv(x,t_1))dx\le  \int_\Omega F(D\bv(x,t_0))dx.
\end{equation}
Comparing \eqref{ConvEq1} and \eqref{ConvEq2}, we deduce 
$$
\Theta=D\psi(\bv_t)
$$
by employing the strict convexity of $\psi$.  Equality now holds in \eqref{ConvEq2} and thus
$$
\liminf_{j\rightarrow\infty}\int^{t_1}_{t_0}\int_\Omega \psi(\bv^{k_j}_t(x,t))dxdt=\int^{t_1}_{t_0}\int_\Omega \psi(\bv_t(x,t))dxdt.
$$
Appealing to the convexity of $\psi$ again, we may pass to yet another subsequence (not relabeled) such that $\bv_t^{k_j}\rightarrow \bv_t$ in $L^p(\Omega\times [0,T]; \R^m).$ 
\end{proof}

\begin{rem}
See Corollary \ref{H2Compactness} for a refinement of Theorem \ref{BlowLem} when $p=2$. 
\end{rem}


\section{Implicit time scheme}\label{ITS}
As noted in the introduction, the existence of solutions of the initial value problem \eqref{mainIVP} has been established \cite{Arai, Colli2, Colli, Mielke}.  These results all employ some sort of compactness to produce a solution by passing to an appropriate limit in an associated implicit time scheme.  We sketch how to argue analogously for the initial value problem \eqref{mainIVP} by making use of Theorem \ref{BlowLem}; for this solution $\bv$, we also show the integral $t\mapsto \int_\Omega\psi^*(D\psi(\bv_t(x,t)))dx$ is nonincreasing as detailed in Proposition \ref{NewDecrease}.  Then we focus on scalar equations and show that while uniqueness may fail, there is always a special weak solution when $m=1$. This was inspired by 
some observations we made in our previous joint work \cite{HyndLindgren}.

\par   A natural way to generate solutions of \eqref{mainIVP} is by the {\it implicit time scheme}: $\bv^0=\bg$
\begin{equation}\label{IFT}
\begin{cases}
D\psi\left(\frac{\bv^k - \bv^{k-1}}{\tau}\right)=\Div DF(D\bv^k), \quad & x\in \Omega\\
\hspace{.85in}\bv^k=0,\quad & x\in \partial \Omega
\end{cases},
\end{equation}
for  $ k=1,\dots, N$.  Here $\tau>0$ and \eqref{IFT} holds in the weak sense: 
\begin{equation}\label{WeakIFT}
\int_\Omega D\psi\left(\frac{\bv^k(x) - \bv^{k-1}(x)}{\tau}\right)\cdot  \bw(x)dx + \int_\Omega DF(D\bv^k(x))\cdot D\bw(x)dx=0, 
\end{equation}
for each $\bw\in W^{1,p}_0(\Omega; \R^m)$ and $k=1,\dots, N$.

Observe that once $\bv^1,\dots, \bv^{k-1}$ are determined, $\bv^k$ can be chosen as the unique minimizer of the strictly convex functional 
$$
W^{1,p}_0(\Omega; \R^m)\ni\bu\mapsto \int_\Omega \left\{F(D\bu) +\tau \psi\left(\frac{\bu - \bv^{k-1}}{\tau}\right)\right\}dx.
$$
Sometimes we write $\bv^k_\tau$ to denote the dependence on the small parameter $\tau$.  Choosing $\bw=\bv^k-\bv^{k-1}$ in \eqref{WeakIFT} and using the convexity of $F$ gives
$$
\int_\Omega D\psi\left(\frac{\bv^k(x) - \bv^{k-1}(x)}{\tau}\right)\cdot  (\bv^k(x) - \bv^{k-1}(x))dx + \int_\Omega F(D\bv^k(x))dx\le \int_\Omega F(D\bv^{k-1}(x))dx. 
$$
Summing over $k=1,\dots, j\le N$ yields 
\begin{equation}\label{DiscreteEnergy}
\sum^j_{k=1}\int_\Omega D\psi\left(\frac{\bv^k(x) - \bv^{k-1}(x)}{\tau}\right)\cdot (\bv^k(x) - \bv^{k-1}(x))dx + \int_\Omega F(D\bv^j(x))dx\le \int_\Omega F(D\bg(x))dx
\end{equation}
which is a discrete version of identity \eqref{EnergyIdentity2}.

\par Let us now deduce a discrete version of Proposition \ref{NewDecrease}. Observe 
\begin{align*}
\int_\Omega \psi^*\left(D\psi\left(\frac{\bv^k(x)-\bv^{k-1}(x)}{\tau}\right)\right)dx&=\int_\Omega\left\{D\psi\left(\frac{\bv^k(x)-\bv^{k-1}(x)}{\tau} \right)\cdot \frac{\bv^k(x)-\bv^{k-1}(x)}{\tau} \right. \\
&\quad \left. - \quad \psi\left(\frac{\bv^k(x)-\bv^{k-1}(x)}{\tau}\right) \right\}dx\\
&=\int_\Omega\left\{-DF(D\bv^k(x))\cdot \frac{D\bv^k(x)-D\bv^{k-1}(x)}{\tau} \right. \\
&\quad \quad \left. - \quad \psi\left(\frac{\bv^k(x)-\bv^{k-1}(x)}{\tau}\right) \right\}dx.
\end{align*}
Thus, for $k =2,\dots, N$, 
\begin{align*}
Q:&=\int_\Omega \psi^*\left(D\psi\left(\frac{\bv^k(x)-\bv^{k-1}(x)}{\tau}\right)\right)dx-\int_\Omega \psi^*\left(D\psi\left(\frac{\bv^{k-1}(x)-\bv^{k-2}(x)}{\tau}\right)\right)dx \\
   & = \int_\Omega\left\{ - D\psi\left(\frac{\bv^{k-1}(x)-\bv^{k-2}(x)}{\tau} \right)\cdot \frac{\bv^{k-1}(x)-\bv^{k-2}(x)}{\tau}  \right. \\
   & \left.\hspace{1in}-DF(D\bv^k(x))\cdot \frac{D\bv^k(x)-D\bv^{k-1}(x)}{\tau} \right. \\
  & \quad +\left.\psi\left(\frac{\bv^{k-1}(x)-\bv^{k-2}(x)}{\tau}\right) -\psi\left(\frac{\bv^k(x)-\bv^{k-1}(x)}{\tau}\right) \right\}dx\\
 & \le  \int_\Omega\left\{ - D\psi\left(\frac{\bv^{k-1}(x)-\bv^{k-2}(x)}{\tau} \right)\cdot \frac{\bv^{k-1}(x)-\bv^{k-2}(x)}{\tau}  \right. \\
   & \left.\hspace{1in}-DF(D\bv^{k-1}(x))\cdot \frac{D\bv^k(x)-D\bv^{k-1}(x)}{\tau} \right. \\
  & \quad +\left.\psi\left(\frac{\bv^{k-1}(x)-\bv^{k-2}(x)}{\tau}\right) -\psi\left(\frac{\bv^k(x)-\bv^{k-1}(x)}{\tau}\right) \right\}dx\\
   & = \int_\Omega\left\{ D\psi\left(\frac{\bv^{k-1}(x)-\bv^{k-2}(x)}{\tau} \right)\cdot \left(\frac{\bv^{k}(x)-\bv^{k-1}(x)}{\tau} - \frac{\bv^{k-1}(x)-\bv^{k-2}(x)}{\tau} \right)  \right. \\
  & \quad +\left.\psi\left(\frac{\bv^{k-1}(x)-\bv^{k-2}(x)}{\tau}\right) -\psi\left(\frac{\bv^k(x)-\bv^{k-1}(x)}{\tau}\right) \right\}dx\\
  &\le 0.
\end{align*}
Here we have used convexity of both $\psi$ and $F$.  Therefore, 
\begin{equation}\label{DiscreteNewDecrease}
\int_\Omega \psi^*\left(D\psi\left(\frac{\bv^k(x)-\bv^{k-1}(x)}{\tau}\right)\right)dx\le \int_\Omega \psi^*\left(D\psi\left(\frac{\bv^{k-1}(x)-\bv^{k-2}(x)}{\tau}\right)\right)dx
\end{equation}
for each $k=2,\dots, N$. 
\par Now set $\tau:=T/N$ and $\tau_k:=k\tau$ for $k=0,1,\dots, N$.  Also define  
\begin{equation}\label{StepApprox}
\begin{cases}
\bv_N(x, t)=\bg(x), \hspace{.23in} t=0\\
\hspace{.6in}                 =\bv^{k}(x), \quad \tau_{k-1}< t \le  \tau_k
\end{cases}
\end{equation}
and 
$$
\bu_N(x, t)  = \bv^{k-1}(x) +\frac{t-\tau_{k-1}}{\tau}(\bv^k(x)-\bv^{k-1}(x)), \quad \tau_{k-1}\le t \le \tau_k.
$$
Note that for $t\in [0,T]\setminus\{\tau_0,\tau_1,\dots,\tau_N\}$, \eqref{WeakIFT} can now be rewritten 
$$
\int_\Omega D\psi\left(\partial_t \bu_N(x,t)\right)\cdot  \bw(x)dx + \int_\Omega DF(D\bv_N(x,t))\cdot D\bw(x)dx=0, 
$$
and inequality \eqref{DiscreteEnergy} implies 
$$
\sup_{N\in \N}\left\{\int^T_0\int_\Omega |\partial_t\bu_N(x,t)|^pdxdt+\max_{0\le t\le T}\int_\Omega |D\bv_N(x,t)|^pdx\right\}<\infty.
$$ 
Moreover, \eqref{DiscreteNewDecrease} gives that
$$
[0,T]\ni t\mapsto\int_\Omega \psi^*(D\psi(\partial_t\bu_N(x,t)))dx
$$
is nonincreasing. 
\par We will omit the proof of the following claim as it is similar to the 
proof of Theorem \ref{BlowLem}.  

\begin{prop}\label{ExistProp}
Assume $T>0$ and $\bg \in W^{1,p}_0(\Omega; \R^m)$. There are subsequences $\{\bv_{N_j}\}_{j\in \N}$ and $\{\bu_{N_j}\}_{j\in \N}$ that both converge to a weak solution $\bv$ as described in 
\eqref{weakConv1}; the subsequence $\{\partial_t\bu_{N_j}\}_{j\in \N}$ also converges to $\bv_t$ as described in \eqref{StrongConv}. Moreover, $\bv$ is a weak solution of the initial value problem \eqref{mainIVP} and 
\begin{equation}\label{RigorNewDecrease}
\int_\Omega\psi^*(D\psi(\bv_t(x,t_2)))dx\le \int_\Omega\psi^*(D\psi(\bv_t(x,t_1)))dx
\end{equation}
for Lebesgue almost every $(t_1,t_2)\in [0,T]\times [0,T]$ with $t_1\le t_2$. 
\end{prop}
\begin{rem}
We emphasize that we do not know if {\it every} weak solution satisfies inequality \eqref{RigorNewDecrease}. However, any solution arising via the implicit scheme will satisfy \eqref{RigorNewDecrease}; this follows from inequality \eqref{DiscreteNewDecrease} and the convergence \eqref{StrongConv}.

\end{rem}

\begin{cor}
Let $\bv$ be a weak solution described in Proposition \ref{ExistProp}. Then
$$
\bv_t\in L^\infty_{loc}((0,T); L^p(\Omega;\R^m)).
$$
\end{cor}
\begin{proof}
Let $\delta\in (0,T)$ and choose a time $t_0\in (0,\delta)$ for which $\bv_t(\cdot, t_0)\in L^p(\Omega;\R^m)$.  By the previous proposition, 
\begin{equation}\label{MontoneTee1}
\int_\Omega\psi^*(D\psi(\bv_t(x,t)))dx\le \int_\Omega\psi^*(D\psi(\bv_t(x,t_0)))dx
\end{equation}
for almost every $t\in [t_0,T]$. By elementary arguments, the bounds \eqref{Coercive} and \eqref{GrowthCond} imply there is a universal constant $C$ such that $|w|^p\le C\left(\psi^*(D\psi(w))+1\right)$ for 
each $w\in \R^m$.  It now follows from \eqref{MontoneTee1} that $\int_\Omega |\bv_t(x,t)|^pdx\le C \int_\Omega\psi^*(D\psi(\bv_t(x,t_0)))dx + C|\Omega|$ for almost every $t\in [t_0,T]$. 
\end{proof}

\par We now consider the implicit time scheme when $m=1$.  Recall in this case the system \eqref{mainIVP} reduces to 
\begin{equation}\label{scalarIVP2}
\begin{cases}
\psi'(v_t)=\Div DF(Dv), &\quad \Omega\times(0,T) \\
\hspace{.33in}v =0, &\quad \partial\Omega\times [0,T)  \\ 
\hspace{.33in}v =g, &\quad \Omega\times\{0\}
\end{cases}.
\end{equation}
Here $\psi\in C^1(\R)$ and $F\in C^1(\R^n)$. Moreover, this equation is a nonlinear parabolic equation for a scalar function $v: \Omega\times (0,T)\rightarrow \R$.  It is then natural to expect 
that the theory of viscosity solutions applies as presented in \cite{CIL, FS}.  

\par To this end, we will additionally assume $p\in [2,\infty)$, $g\in C(\overline{\Omega})\cap  W^{1,p}_0(\Omega)$ and 
\begin{equation}\label{FCtwo}
F\in C^2(\R^n).
\end{equation}
The regularity assumption \eqref{FCtwo} allows us to rewrite the PDE in \eqref{scalarIVP2} as
\begin{equation}\label{scalarPDE}
\psi'(v_t)=\Div DF(Dv)=D^2F(Dv)\cdot D^2v.
\end{equation}
This minor observation will be useful to us when considering viscosity solutions. 
\begin{defn}
$v\in USC(\Omega\times (0,T))$ is a {\it viscosity subsolution} of \eqref{scalarPDE} if 
$$
\psi'(\phi_t(x_0,t_0))\le D^2F(D\phi(x_0,t_0))\cdot D^2\phi(x_0,t_0)
$$
whenever $\phi\in C^\infty(\Omega\times (0,T))$ and $v-\phi$ has a local maximum at $(x_0,t_0)\in \Omega\times (0,T)$. Likewise, $v\in LSC(\Omega\times (0,T))$ is a {\it viscosity supersolution}
if 
$$
\psi'(\phi_t(x_0,t_0))\ge D^2F(D\phi(x_0,t_0))\cdot D^2\phi(x_0,t_0)
$$
whenever $\phi\in C^\infty(\Omega\times (0,T))$ and $v-\phi$ has a local minimum at $(x_0,t_0)\in \Omega\times (0,T)$. $v\in C(\Omega\times (0,T))$ is a {\it viscosity solution} if it is both a viscosity 
sub- and supersolution. 
\end{defn}
Viscosity solutions of the implicit scheme \eqref{IFT} when $m=1$ are defined analogously. We will also make the assumption that there are positive numbers $\theta, \Theta$ for which
\begin{equation}\label{RegAssump}
\theta |M|^{p-2}|\xi|^2\le D^2F(M)\xi\cdot \xi\le \Theta |M|^{p-2}|\xi|^2, \quad M,\xi\in \R^n.
\end{equation}
Under \eqref{FCtwo} and \eqref{RegAssump}, any weak solution $w\in W^{1,p}_0(\Omega)$ of the PDE 
$$
-\Div DF(Dw) = f,  \quad x\in \Omega\\
$$
is necessarily {\it continuous} provided $f\in L^\infty(\Omega)$; see \cite{DB} for precise estimates. 

\par The uniqueness of viscosity solutions of \eqref{scalarIVP2} follows from well known methods (see section 8 of \cite{CIL} for instance). In particular, a proof can be constructed that is similar to the comparison for solutions of the heat equation, so we omit the required argument.  We do remark that the main structural condition needed is the strict monotonicity of  
$\psi'$. In order to prove Theorem \ref{ViscSolnResult}, the main result of this section, we first have to verify a few technical lemmas.
 
\begin{lem}\label{IFTSeqBound}
Let $\{v^1,\dots, v^N\}\subset  W^{1,p}_0(\Omega)$ denote the solution sequence of the implicit scheme \eqref{IFT} for $N\in \N$ and $\tau>0$. Then $v^1,\dots, v^N$ are  viscosity solutions and 
$$
\max_{1\le k \le N}\sup_\Omega|v^k|\le \sup_\Omega |g|.
$$ 
\end{lem}
\begin{proof}
Let us consider the implicit scheme for $k=1$. The PDE 
\begin{equation}\label{kequal1eq}
\psi'\left(\frac{v^1-g}{\tau}\right)=\Div DF(Dv^1)
\end{equation}
admits a comparison principle among weak sub- and supersolutions as $\psi'$ and $DF$ are strictly monotone.  As $\sup_\Omega|g|$ is a supersolution of \eqref{kequal1eq} that is at least as large as $v^1$ on $\partial \Omega$, $v^1\le \sup_\Omega|g|$. Likewise, 
$v^1\ge -\sup_\Omega|g|$. By \eqref{FCtwo} and \eqref{RegAssump}, $v^1\in C(\Omega)$ and $|v^1|\le \sup_\Omega|g|$. By induction, we can make the same conclusion for the weak solutions $\{v^2,\dots, v^N\}$. 

\par Let us now verify $v^1$ is indeed a viscosity solution; here we will follow the approach of \cite{JLM2}.  Assume that $\phi\in C^\infty(\Omega)$ and that $v^1-\phi$ has a strict local maximum at $x_0\in \Omega$.  We assert 
$$
\psi'\left(\frac{v^1(x_0)-g(x_0)}{\tau}\right)\le D^2F(D\phi(x_0))\cdot D^2\phi(x_0)
$$
If not, by continuity there is $\delta>0$ such that 
$$
\begin{cases}
(v^1-\phi)(x)\le (v^1-\phi)(x_0)\\
\psi'\left(\frac{v^1-g}{\tau}\right)>\Div DF(D\phi)
\end{cases}, \quad x\in B_\delta(x_0).
$$
Set 
$$
c:=\max_{\partial B_\delta(x_0)}(v^1-\phi)
$$
and note $c<(v^1-\phi)(x_0).$  Observe
$$
\begin{cases}
-\Div DF(D(\phi+c))\ge - \Div DF(Dv^1), \quad &x\in B_\delta(x_0)\\
\hspace{1in} c+\phi \ge v^1, \quad &  x\in \partial B_\delta(x_0)\
\end{cases}.
$$
Thus, $c+\phi\ge v^1$ in  $B_\delta(x_0)$; in particular, $c\ge (v^1-\phi)(x_0)$ which is a contradiction.  It follows that $v^1$ is viscosity subsolution. The proof that $v^1$ is a supersolution is similar and left to the reader. 
\end{proof}

\begin{cor}\label{LemdiscreteVisc}
Let $N\in \N$. Further assume $\{\phi^0,\phi^1,\dots,\phi^N\}\subset C^\infty(\Omega)$ and $(x_0,k_0)\in\Omega\times\{1,\dots,N\}$ is such that
\begin{equation}\label{discreteVisc}
v^k(x)-\phi^k(x)\le v^{k_0}(x_0)-\phi^{k_0}(x_0) 
\end{equation}
for $x$ in a neighborhood of $x_0$ and $k\in \{k_0-1,k_0\}$. Then 
$$
\psi'\left(\frac{\phi^{k_0}(x_0) - \phi^{k_0-1}(x_0)}{\tau}\right)\le D^2F(D\phi^{k_0}(x_0))\cdot D^2\phi^{k_0}(x_0).
$$
\end{cor}

\begin{proof}
Evaluating the left hand side \eqref{discreteVisc} at $k=k_0$ gives 
$$
\psi'\left(\frac{v^{k_0}(x_0) - v^{k_0-1}(x_0)}{\tau}\right)\le  D^2F(D\phi^{k_0}(x_0))\cdot D^2\phi^{k_0}(x_0),
$$
as $v^k$ is a viscosity solution of \eqref{IFT} for $m=1$.  Evaluating the left hand side of \eqref{discreteVisc} at $x=x_0$ and $k=k_0-1$ gives
$\phi^{k_0}(x_0) - \phi^{k_0-1}(x_0)\le v^{k_0}(x_0) - v^{k_0-1}(x_0)$. The claim follows from the above inequality and the monotonicity of 
$\psi'$. 
\end{proof} 
Let us now define the respective upper and lower limits 
$$
\overline{v}(x,t):=\limsup_{\substack{N\rightarrow\infty\\ (y,s)\rightarrow(x,t)}}v_N(y,s),
$$
$$
\underline{v}(x,t):=\liminf_{\substack{N\rightarrow\infty\\ (v,s)\rightarrow(x,t)}}u_N(y,s)
$$
of the sequence $(v_N)_{N\in\N}$ specified in \eqref{StepApprox}. The functions $\overline{v}, \underline{v}$ are sometimes termed the ``relaxed" limits 
of the sequence $(v_N)_{N\in\N}$ and were introduced by G. Barles and G. Perthame \cite{BarPer, BarPer2} to study convergence properties of viscosity solutions. 
By Lemma \ref{IFTSeqBound}, the sequence $(v_N)_{N\in\N}$ is locally bounded, independently of $N\in \N$.  Thus, $\overline{v}, \underline{v}$ are well defined and finite at each 
$(x,t)\in \Omega\times (0,T)$. Moreover, one checks $\overline{v}, -\underline{v}$ are upper semicontinuous and $\overline{v}= \underline{v}$ 
if and only if $v_N$ converges locally uniformly. The following lemma is proved in \cite{HyndLindgren} and is the last ingredient we will need to prove Theorem \ref{ViscSolnResult}. 
\begin{lem}\label{ApproxPoints}
Assume $\phi\in C^\infty(\Omega\times(0,T))\cap C(\overline{\Omega}\times[0,T])$.  For $N\in \N$ define 
$$
\phi_N(x,t):=
\begin{cases}
\phi(x,0), \quad (x,t)\in \Omega\times \{0\},\\
\phi(x,\tau_k), \quad (x,t)\in\Omega \times(\tau_{k-1}, \tau_k]\quad k=1,\dots,N
\end{cases}.
$$
Suppose $\overline{v}-\phi$ has a strict local maximum at $(x_0,t_0)\in \Omega\times(0,T)$. Then there is $(x_j,t_j)\rightarrow  (x_0,t_0)$ and $N_j\rightarrow \infty$, as $j\rightarrow \infty$, such that 
$v_{N_j} -\phi_{N_j}$ has local maximum at $(x_j,t_j)$. 
\end{lem}

\begin{proof}[Proof of Theorem \ref{ViscSolnResult}] We only show that $\overline{v}$ is a viscosity subsolution of \eqref{scalarPDE}; similar arguments can be used to show  $\underline{v}$ is a supersolution of \eqref{scalarPDE}. 
By the comparison of viscosity solutions, we would then have $\overline{v}\le \underline{v}$. In this case, $\overline{v}=\underline{v}:=v$ is continuous and $v_N$ converges to $v$ locally uniformly. 
The theorem would then follow as each subsequence of $v_N(\cdot,t)$ has a further subsequence converging to a weak solution of \eqref{scalarPDE}  in $L^p(\Omega)$, uniformly in $[0,T]$, by Proposition \ref{ExistProp}.

\par Assume that $\phi\in C^\infty(\Omega\times(0,T))$ and $\overline{v}-\phi$ has a strict local maximum at $(x_0,t_0)\in \Omega\times(0,T)$. By Lemma \ref{ApproxPoints}, there are points 
$(x_j,t_j)$ converging to $(x_0,t_0)$ and $N_j\in \N$ tending to $+\infty$, as $j\rightarrow \infty$, such that $v_{N_j}-\phi_{N_j}$ has a local maximum at $(x_j,t_j)$.  Observe that for each $j\in \N$, 
$t_j\in (\tau_{k_j-1}, \tau_{k_j}]$ for some $k_j\in\{0,1,\dots, N_j\}$. Hence, by the definition of $v_{N_j}$ and $\phi_{N_j}$, 
$$
\Omega\times \{0,1,\dots,N_j\}\ni (x,k)\mapsto v^k(x)- \phi(x,\tau_{k})
$$
has a local maximum at $(x,k)=(x_j, k_j)$.  By Lemma \ref{LemdiscreteVisc}, 
$$
\psi'\left(\frac{ \phi(x_j,\tau_{k_j})-  \phi(x_j,\tau_{k_j-1})}{T/N_j}\right)\le  D^2F(D\phi(x_j,\tau_{k_j}))\cdot D^2\phi(x_j,\tau_{k_j}).
$$
As $\tau_{k_j-1}=\tau_{k_j}-T/N_j$ and $|t_j-\tau_{k_j}|\le T/N_j$ for $j\in \N$, we can send $j\rightarrow \infty$ above, appealing to the smoothness of $\phi$, and arrive at 
$$
\psi'(\phi_t(x_0,t_0))\le D^2F(D\phi(x_0,t_0))\cdot D^2\phi(x_0,t_0).
$$
\end{proof}

\begin{rem}
We anticipate that a version of Theorem \ref{ViscSolnResult} holds for all $p\in (1,\infty)$. In particular, 
we believe the methods described in \cite{JLM2} for the $p$-Laplace equation $-\Div(|Dv|^{p-2}Dv)=0$ in the range $p\in (1,2)$ can be adapted to establish such a generalization. 
\end{rem}


\section{Regularity}\label{RegSect}
We will now consider the interior regularity of weak solutions of the system \eqref{mainPDE}. By weak solutions of \eqref{mainPDE}, and not necessarily the initial value problem \eqref{mainIVP}, we mean
measurable mappings 
$\bv :\Omega\times (0,T)\rightarrow \R^m$ that satisfy 
\begin{equation}\label{NaturalSpace2}
\bv_t\in L^p(\Omega\times(0,T); \R^m), \quad \bv\in L^\infty([0,T]; W^{1,p}(\Omega;\R^m))\nonumber
\end{equation}
and the weak solution condition \eqref{WeakSolnCond}.  We further specialize to the case $p=2$ and assume 
\begin{equation}\label{minPSInow}
\inf_{\R^m}\psi=\psi(0)=0,
\end{equation}
\begin{equation}\label{RegPSIF}
\psi\in C^2(\R^m), F\in C^2(\M^{m\times n}),
\end{equation}
and  
\begin{equation}\label{UnifConv}
\begin{cases}
\alpha |z|^2\le \sum^{m}_{i,j=1}\psi_{w_iw_j}(w)z_iz_j\le A|z|^2, \quad w,z\in \R^m\\
\theta |\xi|^2\le \sum^m_{i,k=1}\sum^n_{j,l=1}F_{M^i_j,M^k_l}(M)\xi^{i}_j\xi^k_l\le \Theta |\xi|^2, \quad M,\xi\in \M^{m\times n}.
\end{cases}
\end{equation}
Here $\alpha, A,\theta,\Theta$ are positive constants.  

\par Observe that when $m=1$, equation \eqref{mainPDE} can be rewritten as 
\begin{equation}\label{misoneEq}
v_t=G(D^2v,Dv)
\end{equation}
where 
$$
G(X,\zeta):=(\psi^*)^{'}(D^2F(\zeta)\cdot X), \quad X=X^t\in \Mnn,\; \zeta\in \R^n.
$$
Under assumptions \eqref{RegPSIF} and \eqref{UnifConv}, the PDE \eqref{misoneEq} is uniformly parabolic. If we assume the third derivatives of $F$ are uniformly bounded,
then the nonlinearity $G$ satisfies 
$$
|G_{\zeta_{i}}(X,\zeta)|\le C|X|, \quad i=1,\dots, n. 
$$
In this case, the results of L. Caffarelli and L. Wang \cite{CaffWang} imply that for any viscosity solution $v$ of \eqref{misoneEq}, $Dv$ and $v_t$ exist and are locally H\"{o}lder continuous. In view of Theorem \ref{ViscSolnResult}, we have the following result.  

\begin{thm}
Assume $m=1$,  \eqref{RegPSIF}, \eqref{UnifConv}, $g\in C(\overline{\Omega})\cap W^{1,2}_0(\Omega)$, and $F\in W^{3,\infty}(\R^n)$. Then the weak solution of \eqref{mainPDE} that is the limit of the sequence $(v_N)_{N\in \N}$ defined in \eqref{CLaan} is continuously differentiable and has locally H\"{o}lder continuous derivatives.
\end{thm}

\par We shall now assume $m>1$ and pursue the partial regularity of weak solutions.  We will argue that the second derivatives of weak solutions are square integrable, and then use the resulting estimate to deduce an improvement of the compactness assertion Theorem \ref{BlowLem} when $p=2$.  This strengthening of the compactness of weak solutions is essential to our partial regularity approach which is based on a ``blow-up" technique.

\subsection{Integral estimates}\label{IntSubsec}
For each weak solution $\bv=(v^1,\dots,v^m)$ of \eqref{mainPDE}, we will denote $D^2\bv$ as the $(\Mnn)^m$ valued mapping
$$
D^2\bv:=(D^2v^1,\dots D^2v^m).
$$ 
As mentioned above, we aim to verify that for each $i\in \{1,\dots,m\}$, $D^2v^i$ is square integrable and then use the resulting estimates to improve upon Theorem \ref{BlowLem} when $p=2$. In this subsection, we will also argue that weak solutions arising as a limit of a subsequence of the implicit time scheme \eqref{IFT} satisfy an additional integral estimate $D\bv_t\in L^2_\text{loc}(\Omega\times (0,T); \Mmn)$.

\par Some notation that will help clarify our arguments are as follows. 
For a given $M=(M^1,\dots,M^m)\in (\Mnn)^m$ and $z\in \R^n$, we will write $Mz\in \Mmn$ for the matrix with $i,j$th components
\begin{equation}\label{MMNnotation}
M^iz\cdot e_j,
\end{equation}
and $Mz\cdot z\in \R^m$ for the vector with $j$th entry 
\begin{equation}\label{MMNnotation2}
M^jz\cdot z.
\end{equation}
Note that in \eqref{MMNnotation} and \eqref{MMNnotation2} the $``\cdot"$ is the usual dot product on $\R^n$.

\begin{prop}
Assume $\psi$ and $F$ satisfy \eqref{UnifConv}, and that $\bv$ is a weak solution of \eqref{mainPDE} on $\Omega\times (0,T)$.  Then $D^2\bv\in L^2_{\text{loc}}(\Omega\times [0,T]; (\Mnn)^m)$. Moreover,
for each open $\Sigma\subset\subset\Omega$ there is a constant $C=C(m,n, \Sigma,A,\theta,\Theta)$ such that
\begin{equation}\label{D2Vbound}
\int^T_0\int_{\Sigma} |D^2\bv(x,t)|^2dxdt\le C\int^T_0\int_\Omega \left(|\bv_t(x,t)|^2+|D\bv(x,t)|^2\right)dxdt.
\end{equation}
\end{prop}
\begin{proof}
Fix a time $t\in (0,T)$ for which $D\psi(\bv_t(\cdot,t))\in L^2(\Omega)$; recall that the set of such times has full Lebesgue measure in $(0,T).$ Since $x\mapsto \bv(x,t)$ satisfies
the uniformly elliptic equation \eqref{mainPDE}, the associated $W^{2,2}_{\text{loc}}(\Omega)$ estimates (Proposition 8.6 in \cite{GiaMar} or Theorem 1, Section 8.3 of \cite{Evans}) imply $D^2\bv(\cdot, t)\in L^2_{\text{loc}}(\Omega; (\Mnn)^m)$ and
$$
\int_{\Sigma} |D^2\bv(x,t)|^2dx\le C_0\int_\Omega\left\{|D\psi(\bv_t(x,t))|^2+|D\bv(x,t)|^2\right\}dx.
$$
Here $C_0=C_0(n, \Sigma,\theta,\Theta)$. The bound \eqref{D2Vbound} now follows from integration in time and employing \eqref{minPSInow} and \eqref{UnifConv}, which imply $|D\psi(w)|=|D\psi(w)-D\psi(0)|\le \sqrt{m}A|w|$.
\end{proof}
Interestingly enough, we can use the above bound to obtain compactness of second derivatives of solutions of \eqref{mainPDE}. We view this an as improvement of Theorem \ref{BlowLem} when $p=2$.  

\begin{cor}\label{H2Compactness}
Assume $\{\bv^k\}_{k\in \N}$ is a sequence of weak solutions of \eqref{mainPDE} on $\Omega\times (0,T)$ converging to another weak solution $\bv$ in the following sense
\begin{equation}\label{StrongConv2}
\begin{cases}
\bv^k\rightarrow \bv \; \text{in}\; L^2([0,T] ; H^1(\Omega;\R^m)) \\\ 
\bv^k_t\rightarrow \bv_t \; \text{in}\; L^2(\Omega\times [0,T];\R^m)
\end{cases}.
\end{equation}
Then there is a subsequence $\{D^2\bv^{k_j}\}_{j\in \N}$ that converges to $D^2\bv$ in $L^2_{\text{loc}}(\Omega\times (0,T);(\Mnn)^m)$. 
\end{cor}
\begin{proof}
By the estimate \eqref{D2Vbound}, $\{D^2\bv^k\}_{k\in \N}$ is bounded in $L^2_{\text{loc}}(\Omega\times(0,T);(\Mnn)^m)$ and thus a subsequence $\{D^2\bv^{k_j}\}_{k\in \N}$ converges weakly to $D^2\bv$. We will now 
argue that this convergence occurs in fact strongly.  Fix $i\in \{1,\dots, n\}$ and set 
$$
\bw^{k} : = \bv^{k}_{x^i}\quad\text{and}\quad \bw : = \bv_{x^i}.
$$ 
Differentiating equation \eqref{mainPDE} with respect to $x_i$, we see $\bw^k$ satisfies the PDE
\begin{equation}\label{wkEqn}
\text{div}(\ba^kD\bw^k)=\bF^k, \quad \Omega\times(0,T)
\end{equation}
 in a weak sense where 
$$
\begin{cases}
\ba^k:=D^2F(D\bv^k)\\
\bF^k:=\partial_{x_i}D\psi(\bv^k_t)
\end{cases}.
$$ 
Recall \eqref{UnifConv} implies
\begin{equation}\label{akbounded}
\theta |\xi|^2\le \ba^k(x,t)\xi\cdot \xi \le \Theta |\xi|^2
\end{equation}
for $\xi\in \Mmn$. Moreover, by our assumption \eqref{RegPSIF} and \eqref{StrongConv2}, there is a subsequence $\{\ba^{k_j}\}_{j\in \N}$ such that
$$
\ba^{k_j}\rightarrow \ba:=D^2F(D\bv)
$$
pointwise in $\Omega\times(0,T)$. By the uniform boundedness of $\ba^k$ \eqref{akbounded} and the interpolation of Lebesgue spaces, this convergence also occurs in $L^p_{\text{loc}}(\Omega\times (0,T))$ for each $1\le p<\infty$.  Also observe that $\bF^k\in L^2([0,T]; H^{-1}(\Omega; \R^m))$ converges to $\bF:=\partial_{x_i}D\psi(\bv_t)$ in $L^2([0,T]; H^{-1}(\Omega; \R^m))$. 

\par Now, $D\bw^{k_j}$ is weakly convergent to $D\bw=D\bv_{x_i}$ in $L^2(\Omega \times (0,T); \Mmn)$. And for any nonnegative $\eta\in C^\infty_c(\Omega\times(0,T)$,
$$
\int^T_0\int_\Omega \eta \ba^kD\bw^k\cdot D\bw^kdxdt\ge \int^T_0\int_\Omega\left(\eta \ba^kD\bw\cdot D\bw + 2\ba^kD\bw\cdot(D\bw^k -D\bw)\eta \right)dxdt.
$$
Hence, 
$$
\liminf_{j\rightarrow \infty}\int^T_0\int_\Omega\eta \ba^{k_j}D\bw^{k_j}\cdot D\bw^{k_j}dxdt \ge \int^T_0\int_\Omega \eta \ba D\bw\cdot D\bw dxdt.
$$
This inequality also follows more generally due to results of A. Ioffe \cite{Ioffe1, Ioffe2}. 

\par Employing  the uniform convexity of the function $\Mmn\ni \xi\mapsto \ba^k\xi\cdot \xi$, we obtain through integrating by parts and equation \eqref{wkEqn} that
\begin{align*}
\int^T_0\int_\Omega \eta \ba^kD\bw \cdot D\bw dxdt &\ge \int^T_0\int_\Omega \eta \left\{\ba^kD\bw^k\cdot D\bw^k + 2\ba^kD\bw^k\cdot(D\bw -D\bw^k) \right . \\
&\quad\quad\quad\left. + \theta|D\bw^k-D\bw|^2 \right\} dxdt \\
&= \int^T_0\int_\Omega \left\{\eta \ba^kD\bw^k\cdot D\bw^k -  2D\eta\cdot \ba^kD\bw^k(\bw -\bw^k) \right. \\
&\left.  \;\;+ \theta|D\bw^k-D\bw|^2 \eta\right\} dxdt + 2 \int^T_0\int_\Omega D\psi(\bv^k_t)\cdot \partial_{x_i}(\eta(\bw - \bw^k))dt.
\end{align*}
\par Observe that 
$$
\lim_{j\rightarrow \infty }\int^T_0\int_\Omega D\psi(\bv^{k_j}_t)\cdot \partial_{x_i}(\eta(\bw - \bw^{k_j}))dt=0,
$$
which follows from the strong convergence of $\bv^k_t$ and the weak convergence of the sequence $\{D\bw^{k_j}\}_{j\in\N}$.  Likewise, 
$$
\lim_{j\rightarrow \infty }\int^T_0\int_\Omega D\eta\cdot \ba^{k_j}D\bw^{k_j}(\bw -\bw^{k_j})dxdt=0.
$$
As a result
\begin{align*}
\int^T_0\int_\Omega \eta \ba D\bw \cdot D\bw dxdt & \ge \liminf_{j\rightarrow \infty}\int^T_0\int_\Omega\eta \ba^{k_j}D\bw^{k_j}\cdot D\bw^{k_j}dxdt  \\
& \hspace{1in} +\theta \liminf_{j\rightarrow \infty}\int^T_0\int_\Omega\eta |D\bw^{k_j}-D\bw|^2 dxdt \\
&\ge \int^T_0\int_\Omega \eta \ba D\bw \cdot D\bw dxdt \\
& \hspace{1in} +\theta \liminf_{j\rightarrow \infty}\int^T_0\int_\Omega\eta |D\bw^{k_j}-D\bw|^2 dxdt.
\end{align*}
Hence, $D\bv^{k_j}_{x^i}\rightarrow D\bv_{x_i}\in L^2_{\text{loc}}(\Omega\times(0,T); \Mmn)$ for each $i\in \{1,\dots,n\}$.  
\end{proof}

Note that under our uniform convexity assumption \eqref{UnifConv} on $F$, the heuristic computation given in the proof of Proposition \ref{NewDecrease} yields 
$$
\frac{d}{dt}\int_\Omega\psi^*(D\psi(\bv_t(x,t)))dx \le -\theta \int_\Omega |D\bv _t(x,t)|^2dx
$$ 
for solutions of the initial value problem \eqref{mainIVP}. 
After integrating in time, this formally implies that we should expect $D\bv_t\in L^2_\text{loc}(\Omega\times (0,T); \Mmn)$; recall condition \eqref{DVTbound} is
the main hypothesis in Theorem \ref{PartialRegThm}. We will now verify that this integrability holds
for each solution arising from the implicit time scheme. 

\begin{prop}\label{DVTProp}
Let $\bv$ be a weak solution as described in Proposition \ref{ExistProp}.  Then $\bv$ satisfies \eqref{DVTbound}. Moreover, there is a constant $C=C(m,n, \theta,\Theta,\alpha,A)$ such that 
\begin{equation}\label{DVTLocalMan}
\int^T_0\int_{\Omega}\eta^2|D\bv_t|^2dxdt \le C \int^T_0\int_{\Omega}(|\eta||\eta_t| +|D\eta|^2)|\bv_t|^2dxdt
\end{equation}
for each $\eta\in C^\I_c(\Omega\times(0,T))$. 
\end{prop}
\begin{proof}
1. We will use the same notation as in section \ref{ITS}. By Proposition \ref{ExistProp}, $\partial_t\bu_{N_j}$ converges to $\bv_t$ in $L^2((0,T);L^2(\Omega; \R^m)).$ Without any loss of generality, 
we assume that this convergence also occurs pointwise (since this is true for a subsequence of $\partial_t\bu_{N_j}$).  Let $\delta\in (0,T)$ and choose $t_0\in (0,\delta)$ for which $
\partial_t\bu_{N_j}(\cdot, t_0)\rightarrow \bv(\cdot,t_0)$ in $L^2(\Omega;\R^m)$ and $t_0/T\notin\Q$.  In particular, $t_0\neq \frac{k}{N}T$ for any $N\in \N$ and $k\in\{1,\dots, N\}$.  With this choice
of $t_0$, there is $k_j\in \{1,\dots, N_j\}$ such that  $t_0\in (\tau_{k_j-1}, \tau_{k_j})$ and
\begin{equation}\label{ConvatTO}
\partial_t\bu_{N_j}(\cdot, t_0)=\frac{\bv^{k_j}-\bv^{k_j-1}}{(T/N_j)}\rightarrow  \bv_t(\cdot,t_0)
\end{equation}
in $L^2(\Omega;\R^m)$ as $j\rightarrow \infty$. 

\par We may now employ the uniform convexity of $F$ to improve inequality \eqref{DiscreteNewDecrease} and arrive at
\begin{equation}\label{DecreaseBound2}
\theta \int_\Omega \frac{|D\bv^k-D\bv^{k-1}|^2}{\tau}dx +
\int_\Omega \psi^*\left(D\psi\left(\frac{\bv^k-\bv^{k-1}}{\tau}\right)\right)dx\le \int_\Omega \psi^*\left(D\psi\left(\frac{\bv^{k-1}-\bv^{k-2}}{\tau}\right)\right)dx
\end{equation}
for $k=2,\dots, N$.  Note that $\psi^*\ge 0$, which follows from \eqref{minPSInow}. Letting $\tau=T/N_j$, using \eqref{ConvatTO} and summing the above inequality from $k=k_j+1$ to $N_j$ gives 
\begin{align}\label{DecreaseBound3}
\theta \int^T_{\tau_{k_j+1}}\int_\Omega|\partial_tD\bu_{N_j}|^2dxdt&=\theta \sum^{N_j}_{k=k_j+1}\int_\Omega \frac{|D\bv^k-D\bv^{k-1}|^2}{(T/N_j)}dx \nonumber \\
&\le\int_\Omega\psi^*\left(D\psi\left(\partial_t\bu_{N_j}(x, t_0)\right)\right)dx-\int_\Omega \psi^*\left(D\psi\left(\frac{\bv^{N_j}-\bv^{N_j-1}}{(T/N_j)}\right)\right)dx \nonumber \\
&\le\int_\Omega\psi^*\left(D\psi\left(\partial_t\bu_{N_j}(x, t_0)\right)\right)dx.
\end{align}
Also observe that as $|t_0-\tau_{k_j}|< T/N_j$,  $\tau_{k_j}<\delta$ for all $j\in \N$ sufficiently large. Combining this observation with \eqref{ConvatTO} and \eqref{DecreaseBound3} provides the bound 
$$
\sup_{j\in \N}\int^T_{\delta}\int_\Omega|\partial_tD\bu_{N_j}|^2dxdt <\infty. 
$$
\par 2. It follows that $\{\partial_tD\bu_{N_j}\}_{j\in \N}$ is bounded in $L^2_\text{loc}(\Omega\times (0,T); \Mmn)$. Without loss of generality, let us suppose 
$$
\partial_tD\bu_{N_j}\rightharpoonup W 
$$ 
in $L^2_\text{loc}(\Omega\times (0,T); \Mmn)$. It is not hard to see that in fact $W=D\bv_t$. Indeed, for $\Psi\in C^\infty_c(\Omega\times (0,T); \R^m)$
\begin{align*}
\int^T_0\int_\Omega W\Psi dxdt & = \lim_{j\rightarrow\infty}\int^T_0\int_\Omega\left(\partial_tD\bu_{N_j}\right)\Psi dxdt \\
&=  \lim_{j\rightarrow\infty}\int^T_0\int_\Omega \bu_{N_j} \left(D\Psi_t\right) dxdt\\
&=  \int^T_0\int_\Omega \bv \left(D\Psi_t\right) dxdt.
\end{align*}
As a result, $D\bv_t\in L^2_\text{loc}(\Omega\times (0,T); \Mmn)$ as claimed. 

\par 3. For $\eta\in C^\infty_c(\Omega\times(0,T))$, we set $\eta^k:=\eta(\cdot, \tau_k)\in C^\infty_c(\Omega)$.  
Choosing $\bw=(\eta^k)^2(\bv^k-\bv^{k-1})$ in \eqref{WeakIFT} and manipulating as we did to derive inequality \eqref{DecreaseBound2}, we find 
\begin{align}\label{LastCSapp}
\theta \int_\Omega(\eta^k)^2 \frac{|D\bv^k-D\bv^{k-1}|^2}{\tau}dx & \le \int_\Omega (\eta^{k-1})^2 \psi^*\left(D\psi\left(\frac{\bv^{k-1}-\bv^{k-2}}{\tau}\right)\right)dx  \nonumber \\
& \; - \int_\Omega (\eta^k)^2\psi^*\left(D\psi\left(\frac{\bv^k-\bv^{k-1}}{\tau}\right)\right)dx \nonumber\\
&\; + \int_\Omega 2\eta^k\left(DF(D\bv^{k-1})-DF(D\bv^{k})\right)\cdot \frac{\bv^k-\bv^{k-1}}{\tau}\otimes D\eta^k\nonumber\\
&\; +\int_\Omega\left((\eta^k)^2-(\eta^{k-1})^2\right)\psi^*\left(D\psi\left(\frac{\bv^{k-1}-\bv^{k-2}}{\tau}\right)\right)dx
\end{align}
for $k=2,\dots, N$.
\par Also note that \eqref{UnifConv} implies $|F_{M^i_j,M^k_l}(M)|\le \Theta$ for each $M\in\Mmn$, $i,k=1,\dots,m$ and $j,l=1,\dots, n$. It follows that
$$
|DF(D\bv^k)-DF(D\bv^{k-1})|\le \sqrt{mn}\Theta|D\bv^k-D\bv^{k-1}|
$$
for $k=2,\dots, N$.  Using this inequality, we can estimate the integrand of the second to last integral on the right hand side of \eqref{LastCSapp} as  
\begin{align}\label{CSest}
\left|2\eta^k\left(DF(D\bv^{k-1})-DF(D\bv^{k})\right)\cdot \frac{\bv^k-\bv^{k-1}}{\tau}\otimes D\eta^k \right|
&\le \frac{\theta}{2}(\eta^k)^2 \frac{|D\bv^k-D\bv^{k-1}|^2}{\tau} \nonumber \\
&\hspace{.3in} + \frac{2mn\Theta^2}{\theta} |D\eta^k|^2 \frac{|\bv^{k-1}-\bv^{k-2}|^2}{\tau} 
\end{align}
and combine this upper bound with the integrand on the left hand side of inequality \eqref{LastCSapp}. 

\par 4. Recall the simple estimate $0\le \psi^*(D\psi(w))=D\psi(w)\cdot w -\psi(w)\le \frac{A}{2}|w|^2$ which follows from \eqref{minPSInow} and \eqref{UnifConv}. Using this bound 
to estimate the last integrand on the right hand side of \eqref{LastCSapp}, employing \eqref{CSest} and summing over $k=2,\dots, N$ gives
\begin{align}\label{DiscreteLocalEnergy}
\sum^{N}_{k=2}\int_\Omega(\eta^k)^2 \frac{|D\bv^k-D\bv^{k-1}|^2}{\tau}dx &\le C\sum^{N}_{k=2}\int_\Omega|D\eta^k|^2\frac{|\bv^k-\bv^{k-1}|^2}{\tau}dx \quad + \nonumber \\
&\quad C\sum^{N}_{k=2}\int_\Omega\left((|\eta^k|+|\eta^{k-1}|)\frac{|\eta^k-\eta^{k-1}|}{\tau}\right)\frac{|\bv^{k-1}-\bv^{k-2}|^2}{\tau}.
\end{align}
Here $C$ only depends on $m,n, \theta, \Theta, a$ and $A$.  

\par Define 
$$
\eta_N(x,t):=
\begin{cases}
0, \quad &(x,t)\in \Omega\times \{0\},\\
\eta^k(x), \quad &(x,t)\in\Omega \times(\tau_{k-1}, \tau_k]\quad k=1,\dots,N.
\end{cases}
$$
As $\eta_N$ converges to $\eta$ uniformly on $\Omega\times(0,T)$ and $\partial_tD\bu_N$ converges to $D\bv_t$ weakly in $L^2_\text{loc}(\Omega\times (0,T); \Mmn)$,  
\begin{align*}
\liminf_{j\rightarrow\infty}\sum^{N_j}_{k=2}\int_\Omega(\eta^{k})^2 \frac{|D\bv^{k}-D\bv^{{k}-1}|^2}{(T/N_j)}dx &=\liminf_{j\rightarrow\infty}\int^T_0\int_{\Omega}(\eta_{N_j})^2|\partial_tD\bu_{N_j}|^2dxdt \\
&\ge \int^T_0\int_{\Omega}\eta^2|D\bv_t|^2dxdt.
\end{align*}
\par Likewise, we can argue
$$
\lim_{j\rightarrow\infty}\sum^{N_j}_{k=2}\int_\Omega|D\eta^k|^2\frac{|\bv^k-\bv^{k-1}|^2}{(T/N_j)}dx=\int^T_0\int_{\Omega}|D\eta|^2|\bv_t|^2dxdt
$$
and
$$
\lim _{j\rightarrow\infty}\sum^{N_j}_{k=2}\int_\Omega\left((|\eta^k|+|\eta^{k-1}|)\frac{|\eta^k-\eta^{k-1}|}{(T/N_j)}\right)\frac{|\bv^{k-1}-\bv^{k-2}|^2}{(T/N_j)}dx=\int^T_0\int_{\Omega}2|\eta||\eta_t||\bv_t|^2dxdt.
$$
Consequently, we may send $N=N_j\rightarrow \infty$ in \eqref{DiscreteLocalEnergy} and conclude \eqref{DVTLocalMan}. 
\end{proof} 
\begin{rem}
Inequality \eqref{DVTLocalMan} holds by direct computation for any smooth solution of \eqref{mainPDE}. However, we are only able to rigorously justify this estimate
for solutions arising from the implicit time scheme. 
\end{rem}

\subsection{Partial regularity}
We now pursue partial regularity for weak solutions of the system \eqref{mainPDE} that satisfy the integrability condition \eqref{DVTbound} when $F(M)=\frac{1}{2}|M|^2$. This corresponds to the PDE
\begin{equation}\label{PDEex2}
D\psi(\bv_t)=\Delta \bv, \quad \Omega\times(0,T).
\end{equation}
In particular, our arguments will at least apply to any weak solutions obtained via the implicit time scheme as described in Proposition \ref{ExistProp}. 
So we assume for the remainder of this section that $\bv$ is such a weak solution.   We also remark that while various methods used to obtain partial regularity for parabolic systems do not directly apply, 
we did learn a lot from consulting works such as \cite{DMS, DM, Giaq, GiaqGiu,GIaStr,GiuMir, MaxT}.  In particular, our work was heavily influenced by the references \cite{EvansWeak, Giaq, GiaMar}. 

\par  In what follows, we denote a parabolic cylinder of radius $r>0$ centered at $(x,t)$ as
$$
Q_r(x,t):=B_r(x)\times(t-r^2/2, t+r^2/2)
$$
and the average of a mapping $\bw$ over $Q_r=Q_r(x,t)$ as
$$
\bw_{Q_r}=\ffint\;\bw:=\frac{1}{|Q_r|}\iint_{Q_r}\bw(y,s)dyds.
$$
A quantity that will be of great utility to us is 
\begin{align}\label{localEE}
E(x,t,r)&:=\ffint|\bv_t - (\bv_t)_{Q_r}|^2 dyds+ \ffint\;\; \left|\frac{D\bv - (D\bv)_{Q_r}- (D^2\bv)_{Q_r}(y-x) }{r}\right|^2dyds \nonumber \\
&\hspace{1.5in}+ \ffint|D^2\bv - (D^2\bv)_{Q_r}|^2 dyds
\end{align}
which is defined for $Q_r(x,t)\subset \Omega\times(0,T)$  and $r>0$.  We remark that the notation for the
product $(D^2\bv)_{Q_r}(y-x)$ in \eqref{localEE} is specified in \eqref{MMNnotation}.  Our first task is to verify an important
decay property of $E$. 

\begin{lem}\label{BlowUpLemma}
Assume $\bv$ is a solution of \eqref{PDEex2}. For each $L>0$, there are $\epsilon, \vartheta, \rho\in (0,1/2)$ for which 
\begin{equation}\label{BlowLemHypoth}
\begin{cases}
Q_r(x,t)\subset \Omega\times(0,T), \quad r<\rho\\
|(\bv_t)_{Q_r}|, |(D^2\bv)_{Q_r}|\le L\\
E(x,t,r)\le \epsilon^2
\end{cases} \nonumber
\end{equation}
implies 
$$
E(x,t,\vartheta r)\le \frac{1}{2}E(x,t,r).
$$
\end{lem}

\begin{proof}
1. We will argue in order to obtain a contradiction. If the assertion fails to hold, there is a number $L_0>0$ and sequences $(x_k,t_k)\in\Omega\times(0,T)$, $\epsilon_k\rightarrow 0$, $\vartheta_k\equiv\vartheta$ (selected below), and
$r_k\rightarrow 0$ as $k\rightarrow +\infty$  such that
\begin{equation}\label{ContraUno}
\begin{cases}
Q_{r_k}(x_k,t_k)\subset \Omega\times(0,T)\\
|(\bv_t)_{Q_{r_k}}| + |(D^2\bv)_{Q_{r_k}}|\le L_0\\
E(x_k,t_k,r_k)=\epsilon_k^2
\end{cases}
\end{equation}
while
\begin{equation}\label{ContradictionMaybe}
E(x_k,t_k,\vartheta r_k)> \frac{1}{2}\epsilon_k^2.
\end{equation}

\par Define the sequence of mappings
\begin{equation}\label{BlowupSeq}
\bv^k(y,s):=\frac{\bv(x_k+r_ky,t_k + r_k^2s) - (\bv)_{Q_{r_k}} - (\bv_t)_{Q_{r_k}}r^2_ks -(D\bv)_{Q_{r_k}}r_ky -(D^2\bv)_{Q_{r_k}}r_k^2y\cdot y}
{\epsilon_k r^2_k},\nonumber
\end{equation}
where the notation for the product $(D^2\bv)_{Q_{r_k}}y\cdot y$ is specified in \eqref{MMNnotation2}.
As $E(x_k,t_k,r_k)=\epsilon_k^2$, where $E$ is defined above in \eqref{localEE}, $\bv^k$ satisfies 
\begin{equation}\label{vkkbounds}
\fffint|\bv^k_t|^2dyds + \fffint|D\bv^k|^2dyds +\fffint|D^2\bv^k|^2dyds  =1
\end{equation}
for each $k\in \N$.  Moreover, $\bv^k$ is a weak solution of the PDE
\begin{equation}\label{BlowEqn}
D\psi(a_k +\epsilon _k\bv^k_s)=\text{tr}M_k +\epsilon_k \Delta\bv^k.
\end{equation}
Here $a_k:= (\bv_t)_{Q_{r_k}}\in \R^n$, $M_k:= (D^2\bv)_{Q_{r_k}}\in (\Mnn)^m$ and $\text{tr}M_k:=\sum^m_{i=1}\text{tr}\left[(D^2v^i)_{Q_{r_k}}\right]e_i\in \R^m$. By \eqref{ContraUno}, these sequences are bounded independently of $k\in \N$.  Without any loss of generality, we assume $a_k\rightarrow a\in \R^n$ and $M_k\rightarrow M\in (\Mnn)^m$. 

\par 2. Setting 
$$
\psi_k(w):=\frac{\psi(a_k+\epsilon_k w)-\psi(a_k)-D\psi(a_k)\cdot \epsilon_k w}{\epsilon_k^2}
$$
and 
$$
\beta_k:=\frac{\text{tr}M_k-D\psi(a_k)}{\epsilon_k}\in \R^m
$$
allows us to rewrite the PDE \eqref{BlowEqn} as 
\begin{equation}\label{BlowEqnRewrit}
D\psi_k(\bv^k_s)=\Delta\bv^k+\beta_k. 
\end{equation}
Observe $\psi_k$ is uniformly convex and satisfies the same bounds as $\psi$ in \eqref{UnifConv}. Moreover, 
\begin{equation}\label{ConvPsiPsiStark}
\begin{cases}
\psi_k(w)\rightarrow \frac{1}{2}D^2\psi(a)w\cdot w\\
D\psi_k(w)\rightarrow D^2\psi(a)w
\end{cases}, \quad w\in \R^m
\end{equation}
as $k\rightarrow \infty$.  

\par From \eqref{vkkbounds}, we also have the following convergence: there is a $\bv\in L^2\left([-\frac{1}{2},\frac{1}{2}]; H^1(B_1;\R^m)\right)$ with $\bv_s\in L^2(Q_1;\R^m)$ and $D^2\bv\in L^2(Q_1;(\Mnn)^m)$ such that there is a subsequence $\{\bv^{k_j}\}_{j\in \N}$ satisfying
\begin{equation}\label{AlotOfWeak}
\begin{cases}
\bv^{k_j}\rightarrow \bv\quad \text{in}\quad C([-\frac{1}{2},\frac{1}{2}]; L^2(B_1;\R^m))\\
\bv^{k_j}(\cdot, s)\rightharpoonup \bv(\cdot, s)\quad \text{in}\quad H^1(B_1;\R^m)), \quad |s|\le 1/2\\
\bv^{k_j}_s\rightharpoonup \bv_s\quad \text{in}\quad L^2(Q_1;\R^m)\\
D\psi_k(\bv^{k_j}_s)\rightharpoonup D^2\psi(a)\bv_s\quad \text{in}\quad L^2(Q_1;\R^m)\\
D^2\bv^{k_j}\rightharpoonup D^2\bv\quad \text{in}\quad L^2(Q_1;(\Mnn)^m)\\
\end{cases}.
\end{equation}
Moreover,
\begin{equation}\label{WeakL22bounds}
\fffint|\bv_t|^2dyds + \fffint|D\bv|^2dyds+ \fffint|D^2\bv|^2dyds \le 1.
\end{equation}
\par The weak formulation of \eqref{BlowEqnRewrit} is 
$$
\int_{B_1}\beta_k\cdot \phi(y)dy=\int_{B_1}D\psi_k(\bv^{k}_s(y,s))\cdot \phi(y)dy+\int_{B_1}D\bv^{k}(y,s)\cdot D\phi(y)dy
$$
for each $\phi\in H^1_0(B_1)$ and almost every $|s|<1/2$. It is not hard to see that $\beta_k$ is necessarily bounded.
For instance, we can choose $\phi=\beta_k\eta(\cdot,s)$, for $\eta\in C^\infty_c(Q_1)$ with $\int_{Q_1}\eta=1$ and integrate
the above identity over $s\in [-\frac{1}{2},\frac{1}{2}]$ to obtain 
$$
|\beta_k|^2\le  C \max\{|\eta|_{L^\infty},|D\eta|_{L^\infty}\} |\beta_k|\left\{\iint_{Q_1}\left(|\bv^{k}_s|+|D\bv^{k}|\right)dyds\right\}.
$$
In view of the weak convergence assertions \eqref{AlotOfWeak}, we may assume without loss of generality that $\beta_{k_j}\rightarrow \beta$ in $\R^m$. Passing to the 
limit $k=k_j\rightarrow \infty$ in \eqref{BlowEqnRewrit} gives that $\bv$ solves the linear parabolic equation 
\begin{equation}\label{BlowEqnLimit}
D^2\psi(a)\bv_s=\Delta\bv+\beta. 
\end{equation}
In particular, $\bv\in C^\infty(Q_1; \R^m)$. 

\par 3. Let us now study further the compactness properties of the sequence $\{\bv^{k_j}\}_{j\in \N}$. Again let  $\eta\in C^\infty_c(Q_1)$ be nonnegative. Integrating by 
parts and employing \eqref{BlowEqnRewrit} gives
\begin{align*}
\int_{Q_1}\eta|D\bv|^2dyds &=\int_{Q_1}\left\{\eta|D\bv^{k_j}|^2 +2\eta D\bv^{k_j}\cdot(D\bv-D\bv^{k_j})+\eta|D\bv-D\bv^{k_j}|^2 \right\}dyds\\
&=\int_{Q_1}\left\{\eta|D\bv^{k_j}|^2 -2D\eta\cdot D\bv^{k_j}(\bv-\bv^{k_j})-2\eta (D\psi(\bv^{k_j}_s)-\beta_k)\cdot(\bv-\bv^{k_j}) \right. \\
& \hspace{2in}\left. + \eta|D\bv-D\bv^{k_j}|^2 \right\}dyds.
\end{align*}
Using that $\bv^{k_j}\rightarrow \bv$ in $L^2(Q_1;\R^m)$ and the weak convergence of $D\bv^{k_j}$ in $L^2(Q_1; \Mmn)$, we send $j\rightarrow \infty$ above to get
\begin{align*}
\int_{Q_1}\eta|D\bv|^2dyds &\ge \liminf_{k\rightarrow\infty}\int_{Q_1}\eta|D\bv^{k_j}|^2 + \liminf_{k\rightarrow \infty}\int_{Q_1}\eta|D\bv-D\bv^{k_j}|^2dyds \nonumber \\
&\ge \int_{Q_1}\eta|D\bv|^2dyds+ \liminf_{k\rightarrow \infty}\int_{Q_1}\eta|D\bv-D\bv^{k_j}|^2dyds.
\end{align*}
It follows, that $D\bv^{k_j}\rightarrow D\bv$ in $L^2_\text{loc}(Q_1;\Mmn)$. Without loss of generality, we may assume $D\bv^{k_j}(\cdot,s)\rightarrow D\bv(\cdot,s)$ in $L^2(B_1; \Mmn)$
for almost every $|s|<1/2$ (since this convergence occurs for a subsequence). 

\par 4. Differentiating the integral $\int_{B_1}\frac{1}{2}\eta(y,s)|D\bv^{k_j}(y,s)|^2dy$ with respect to $s$, integrating by parts and making use of \eqref{BlowEqnRewrit} gives the energy identity:
\begin{align*}
\int^s_{-1/2}\int_{B_1}\eta(\psi_k(\bv^{k_j}_s)+\psi^*_k(D\psi_k(\bv^{k_j}_s)))dyds - \int^s_{-1/2}\int_{B_1}\eta \bv^{k_j}_s\cdot \beta_kdyds \\
 + \int_{B_1}\frac{1}{2}\eta(y,s)|D\bv^{k_j}(y,s)|^2dy= \int^s_{-1/2}\int_{B_1}\left\{ \frac{\eta_s}{2} |D\bv^{k_j}|^2 -D\eta\cdot D\bv^{k_j}\bv^{k_j}_s \right\}dyds
\end{align*}
for $s\in [-1/2,1/2]$.  Here we used the equality $D\psi_k(w)\cdot w=\psi_k(w)+\psi_k^*(D\psi_k(w))$. Assuming $s$ is a time where $D\bv^{k_j}(\cdot,s)\rightarrow D\bv(\cdot,s)$ in $L^2(B_1; \Mmn)$, we can send $j\rightarrow \infty$ above to arrive at
\begin{align}\label{IfEqualThenCompact}
\int^s_{-1/2}\int_{B_1}\eta\left(\frac{1}{2}D^2\psi(a)\bv_s\cdot \bv_s+\frac{1}{2}D^2\psi(a)\bv_s\cdot \bv_s\right) dyds - \int^s_{-1/2}\int_{B_1}\eta \bv_s\cdot \beta dyds \nonumber \\
 + \int_{B_1}\frac{1}{2}\eta(y,s)|D\bv(y,s)|^2dy\le \int^s_{-1/2}\int_{B_1}\left\{ \frac{\eta_s}{2} |D\bv|^2 -D\eta\cdot D\bv\bv_s \right\}dyds.
\end{align}
Note in the inequality above, we employed the convergence \eqref{ConvPsiPsiStark} and \eqref{AlotOfWeak} to conclude
$$
\begin{cases}
\liminf_{j\rightarrow \infty}\iint_{Q_1}\psi_k(\bv^{k_j}_s)\eta dyds  \ge \iint_{Q_1}\frac{1}{2}D^2\psi(a)\bv_s\cdot \bv_s \eta dyds & \\
\liminf_{j\rightarrow \infty}\iint_{Q_1}\psi^*_k(D\psi_k(\bv^{k_j}_s))\eta dyds \ge \iint_{Q_1}\frac{1}{2}D^2\psi(a)\bv_s\cdot \bv_s \eta dyds &
\end{cases}.
$$

\par Performing a similar computation with  $\int_{B_1}\frac{1}{2}\eta(y,s)|D\bv(y,s)|^2dy$ and using that $\bv$ satisfies \eqref{BlowEqnRewrit} gives 
\begin{align*}
\int^s_{-1/2}\int_{B_1}\eta D^2\psi(a)\bv_s\cdot \bv_s dyds - \int^s_{-1/2}\int_{B_1}\eta \bv_s\cdot \beta dyds \\
 + \int_{B_1}\frac{1}{2}\eta(y,s)|D\bv(y,s)|^2dy= \int^s_{-1/2}\int_{B_1}\left\{ \frac{\eta_s}{2} |D\bv|^2 -D\eta\cdot D\bv\bv_s \right\}dyds.
\end{align*}
Upon comparing this equality with \eqref{IfEqualThenCompact}, and using the uniform convexity of the sequence $\{\psi_k\}_{k\in \N}$, we find
$$
\begin{cases}
\liminf_{j\rightarrow \infty}\iint_{Q_1}\psi_k(\bv^{k_j}_s)\eta dyds =\iint_{Q_1}\frac{1}{2}D^2\psi(a)\bv_s\cdot \bv_s \eta dyds\\
\liminf_{j\rightarrow \infty}\iint_{Q_1}\psi^*_k(D\psi_k(\bv^{k_j}_s))\eta dyds=\iint_{Q_1}\frac{1}{2}D^2\psi(a)\bv_s\cdot \bv_s \eta dyds
\end{cases}.
$$
In particular, after passing to another subsequence, we have $\bv^{k_j}_s\rightarrow \bv_s$ in $L^2_{\text{loc}}(Q_1;\R^m)$. Furthermore, we may now apply Proposition \ref{H2Compactness} 
to deduce 
\begin{equation}\label{H2CompactnessUse}
D^2\bv^{k_j}\rightarrow D^2\bv \quad \text{in}\quad L^2_{\text{loc}}(Q_1;(\Mmn)^m).
\end{equation}
Of course \eqref{H2CompactnessUse} is only guaranteed to hold for a subsequence that we will not relabel. 

\par 5. Multiplying equation \eqref{BlowEqnLimit} by $\eta \bv$ and integrating by parts gives 
the identity 
\begin{align*}
\int_{B_1}\eta(y,s)\frac{1}{2}D^2\psi(a)\bv(y,s)\cdot \bv(y,s)dy +\int^s_{-1/2}\int_{B_1}\eta |D\bv|^2 dyd\tau =\\
+\int^s_{-1/2}\int_{B_1}\left( \eta_s\frac{1}{2}D^2\psi(a)\bv\cdot\bv +\Delta \eta 
\frac{1}{2}|\bv|^2 +\eta\bv\cdot\beta \right)dyd\tau.
\end{align*}
In particular, every space-time derivative of $\bv$ will satisfy \eqref{BlowEqnLimit} and the above identity without terms involving $\beta$.  As a result, we can bound each higher 
space-time derivative of $\bv$ in terms of lower derivatives, which ultimately can be bounded by a universal constant in view of the bounds \eqref{WeakL22bounds}. 
Consequently, there is a constant $C$ depending only $\alpha, A$ in \eqref{UnifConv} such that 
\begin{align*}
\fthetaint|\bv_s-(\bv_s)_{Q_{\vartheta}}|^2dyds + \fthetaint\;\;\left|\frac{D\bv-(D\bv)_{Q_\vartheta}-(D^2\bv)_{Q_\vartheta}y}{\vartheta}\right|^2dyds \\
\hspace{1in}+\fthetaint|D^2\bv-(D^2\bv)_{Q_{\vartheta}}|^2dyds\le C\vartheta^2.
\end{align*}
We choose $\vartheta_k\equiv\vartheta$ so small that $C\vartheta^2\le 1/4$. 

\par By the compactness established above (in particular \eqref{H2CompactnessUse}), we have for all sufficiently large $j$ 
\begin{align*}
\fthetaint|\bv^{k_j}_s-(\bv^{k_j}_s)_{Q_{\vartheta}}|^2dyds + \fthetaint\;\;\left|\frac{D\bv^{k_j}-(D\bv^{k_j})_{Q_\vartheta}-(D^2\bv^{k_j})_{Q_\vartheta}y}{\vartheta}\right|^2dyds\\
\hspace{1in}+\fthetaint|D^2\bv^{k_j}-(D^2\bv^{k_j})_{Q_{\vartheta}}|^2dyds \le \frac{3}{8}.
\end{align*}
However, it is readily verified that inequality \eqref{ContradictionMaybe} implies that the left hand side of the inequality above is larger than 1/2 for all $k\in \N$. 
Therefore, we have the sought after contradiction.   
\end{proof}
\begin{rem} Generalizing the above argument to equation \eqref{mainPDE} involves studying large $k$ limits of the system 
$$
D\psi(a_k+\epsilon_k\bv^k_s)=\frac{1}{r_k}\Div_y DF(\xi_k+r_k M_k y +\epsilon_k r_k D\bv^k).
$$ 
Here $\xi_k:=(D\bv)_{Q_{r_k}}$.    
\end{rem}
We now seek to iterate Lemma \ref{BlowUpLemma}. First let us recall a basic fact about the decay of averages of $\bv_t$ and 
$D^2\bv$.  Observe for $\tau\in (0,1]$ and $Q_r=Q_r(x,t)\subset \Omega\times(0,T)$, 
\begin{align}\label{IterationVT}
|(\bv_t)_{Q_{\tau r}}- (\bv_t)_{Q_r}|&\le \left(\intQrtau \; |\bv_t - (\bv_t)_{Q_r}|^2\right)^{1/2}\nonumber\\
&\le \frac{1}{\tau^{n/2+1}}\left(\ffint \; |\bv_t - (\bv_t)_{Q_r}|^2\right)^{1/2} \nonumber \\
&\le \frac{1}{\tau^{n/2+1}}E(x,t,r)^{1/2}.
\end{align}
Likewise, 
\begin{equation}\label{IterationD2V}
\left|(D^2\bv)_{Q_{\tau r}}- (D^2\bv)_{Q_r}\right|\le \frac{1}{\tau^{n/2+1}}E(x,t,r)^{1/2}.
\end{equation}
\begin{cor}
Assume $\bv$ is a solution of \eqref{PDEex2}.  Let $L>0$ and select $\epsilon,\vartheta,\rho$ as in Lemma \ref{BlowUpLemma}. If
\begin{equation}\label{IterAssump}
\begin{cases}
Q_r(x,t)\subset \Omega\times(0,T), \quad r<\rho\\
|(\bv_t)_{Q_r}|, |(D^2\bv)_{Q_r}|< \frac{1}{2}L\\
E(x,t,r)< \epsilon_1^2
\end{cases},
\end{equation}
where $\epsilon_1:=\min\left\{\epsilon, \sqrt{\frac{\vartheta^{n/2+1}}{2}L}\right\}$, then 
\begin{equation}\label{InterationK}
\begin{cases}
|(\bv_t)_{Q_{\vartheta^kr} }|, |(D^2\bv)_{Q_{\vartheta^k r}}|< L\\
E(x,t,\vartheta^{k}r)\le \frac{1}{2^{k}}E(x,t,r)
\end{cases}
\end{equation}
for each $k\in \N$.
\end{cor}
\begin{proof}
We will argue by induction. The case $k=1$ was verified in the previous lemma and inequalities \eqref{IterationVT} and \eqref{IterationD2V}. Let us now assume that \eqref{InterationK} holds for each $k=1,2,\dots, j\ge 1$. By \eqref{IterationVT}, 
\begin{align*}
|(\bv_t)_{Q_{\vartheta^{j+1}r} }| & \le \sum^{j-1}_{k=0}|(\bv_t)_{Q_{\vartheta^{k+1}r} }-(\bv_t)_{Q_{\vartheta^{k}r} }|+|(\bv_t)_{Q_{ r} }|
\\
& \le \sum^{j+1}_{k=1}\frac{1}{\vartheta^{n/2+1}}E(x,t,\vartheta^k r)+|(\bv_t)_{Q_{ r} }|\\
& < \sum^{j+1}_{k=1}\frac{1}{\vartheta^{n/2+1}}\frac{1}{2^k}E(x,t,r)+\frac{1}{2}L\\
&\le \frac{\epsilon_1^2}{\vartheta^{n/2+1}}+\frac{1}{2}L\\
&\le L.
\end{align*}
Likewise, we employ \eqref{IterationD2V} to conclude $|(D^2\bv)_{Q_{\vartheta^{j+1}r} }|<L$. As 
$$
E(x,t,\vartheta^{j}r)\le \frac{1}{2^{j}}E(x,t,r)\le \frac{1}{2^j}\epsilon_1^2<\epsilon^2,
$$
the previous lemma implies 
$$
E(x,t,\vartheta^{j+1}r)=E(x,t,\vartheta(\vartheta^{j}r))\le \frac{1}{2}E(x,t,\vartheta^{j}r)\le\frac{1}{2^{j+1}}E(x,t,r).
$$
This verifies the claim. 
\end{proof}
\begin{cor}
Assume $\bv$ is a solution of \eqref{PDEex2}. Let $L>0$ and suppose there are $(x,t)\in \Omega\times(0,T)$ and $r>0$ as in \eqref{IterAssump} of the previous corollary.  Then there exist constants $C\ge 0$, 
$\rho_1\in (0,\rho)$, $\alpha\in (0,1)$ depending on $L$ and a neighborhood $O\subset\Omega\times(0,T)$ of $(x,t)$ such that 
\begin{equation}\label{DecayofE}
E(y,s,R)\le CR^\alpha, \quad R\in (0,\rho_1),\quad (y,s)\in O. 
\end{equation}
\end{cor}
\begin{proof}
1. We first establish a few useful assertions. Let $R<r$ and choose $k\in \N$ so that $\vartheta^{k+1}r<R\le \vartheta^{k}r$. Observe that for any $f\in L^2_\text{loc}(\Omega\times(0,T))$ 
\begin{align}\label{IntegralDEC}
\left(\ffRint|f-f_{Q_R}|^2\right)^{1/2}&\le \left(\ffRint|f-f_{Q_{\vartheta^{k}r}}|^2\right)^{1/2}+|f_{Q_{\vartheta^{k}r}}-f_{Q_R}| \nonumber \\
&\le \frac{1}{\vartheta^{n/2+1}} \left(\intQvarkr|f-f_{Q_{\vartheta^{k}r}}|^2\right)^{1/2}+|f_{Q_{\vartheta^{k}r}}-f_{Q_R}|\nonumber \\
&\le \frac{1}{\vartheta^{n/2+1}} \left(\intQvarkr|f-f_{Q_{\vartheta^{k}r}}|^2\right)^{1/2}+ \left(\ffRint|f-f_{Q_{\vartheta^{k}r}}|^2\right)^{1/2}\nonumber \\
&\le \frac{2}{\vartheta^{n/2+1}} \left(\intQvarkr|f-f_{Q_{\vartheta^{k}r}}|^2\right)^{1/2}.
\end{align}
All parabolic cylinders above and below are assumed to be centered at $(x,t)$. 

\par 2. Now assume additionally that $Df\in L^2_\text{loc}(\Omega\times(0,T);\R^n)$. Note that $f -  (Df)_{Q_R}\cdot (y-x)$ has the same average of $f$ over {\it any} cylinder included in 
$\Omega\times(0,T)$ that is also centered at $(x,t)$. The above computation applied to $f -  (Df)_{Q_R}\cdot (y-x)$ 
implies
\begin{align}\label{IntegralDEC2}
I&:=\left(\ffRint\;\;\left|\frac{f-f_{Q_R}-(Df)_{Q_R}\cdot (y-x)}{R}\right|^2\right)^{1/2} \nonumber \\
&= \left(\ffRint\;\;\left|\frac{\left[f-(Df)_{Q_R}\cdot (y-x)\right]-f_{Q_R}}{R}\right|^2\right)^{1/2} \nonumber \\
&\le \frac{2}{\vartheta^{n/2+1}} \left(\intQvarkr\;\;\;\;\left|\frac{\left[f-(Df)_{Q_R}\cdot (y-x)\right]-f_{Q_{\vartheta^kr}}}{R}\right|^2\right)^{1/2} \nonumber \\
&= \frac{2}{\vartheta^{n/2+1}} \left(\intQvarkr\;\;\;\;\left|\frac{f-f_{Q_{\vartheta^kr}}-(Df)_{Q_R}\cdot (y-x) }{R}\right|^2\right)^{1/2}\nonumber \\
&\le \frac{2}{\vartheta^{n/2+2}} \left(\intQvarkr\;\;\;\;\left|\frac{f-f_{Q_{\vartheta^kr}}-(Df)_{Q_R}\cdot (y-x) }{\vartheta^kr}\right|^2\right)^{1/2} \nonumber \\
&\le \frac{2}{\vartheta^{n/2+2}} \left(\intQvarkr\;\;\;\;\left|\frac{f-f_{Q_{\vartheta^kr}}-(Df)_{Q_{\vartheta^kr}}\cdot (y-x) }{\vartheta^kr}\right|^2\right)^{1/2} \nonumber\\
&\quad \quad + \frac{2}{\vartheta^{n/2+2}}|(Df)_{Q_{\vartheta^kr}}-(Df)_{Q_R}| \nonumber \\
& \le \frac{2}{\vartheta^{n/2+2}} \left(\intQvarkr\;\;\;\;\left|\frac{f-f_{Q_{\vartheta^kr}}-(Df)_{Q_{\vartheta^kr}}\cdot (y-x) }{\vartheta^kr}\right|^2\right)^{1/2}\nonumber \\
& \quad \quad + \frac{2}{\vartheta^{n/2+2}}\frac{1}{\vartheta^{n/2+1}}\left(\intQvarkr|Df-(Df)_{Q_{\vartheta^{k}r}}|^2\right)^{1/2}\nonumber\\
& \le \frac{2}{\vartheta^{n+3}}\left\{\left(\intQvarkr\;\;\;\;\left|\frac{f-f_{Q_{\vartheta^kr}}-(Df)_{Q_{\vartheta^kr}}\cdot (y-x) }{\vartheta^kr}\right|^2\right)^{1/2}\nonumber \right.\\
& \quad \quad\quad  \left. + \left(\intQvarkr|Df-(Df)_{Q_{\vartheta^{k}r}}|^2\right)^{1/2}\right\}.
\end{align}
\par 3. Letting $f=v^i_t, D^2v^i$ in \eqref{IntegralDEC} and summing over $i=1,\dots, m$ and letting $f=v^i_{x_j}$ in  \eqref{IntegralDEC2} and summing over $i=1,\dots,m, j=1,\dots,n$ gives
\begin{align*}
E(x,t,R)&\le \frac{c_0}{\vartheta^{2(n+3)}}E(x,t,\vartheta^k r)
\end{align*}
for some universal constant $c_0$. Applying the previous corollary
\begin{align*}
E(x,t,R)&\le \frac{c_0}{\vartheta^{n+3}} \frac{1}{2^k}E(x,t,r)\\
& \le \frac{c_0\epsilon_1^2}{\vartheta^{2(n+3)}} \frac{1}{2^k}\\
&\le \frac{2c_0\epsilon_1^2}{\vartheta^{2(n+3)}} e^{-(k+1) \log 2}\\
&\le \frac{2c_0\epsilon_1^2}{\vartheta^{2(n+3)}} \left(\frac{R}{r}\right)^\frac{\log 1/2}{\log\vartheta}.
\end{align*}
As $\Omega\times(0,T)\ni (y,s)\mapsto (\bv_t)_{Q_r(y,s)}$, $(D^2\bv)_{Q_r(y,s)}$, and $E(y,s,r)$ are continuous, there is a $\rho_1>0$ and a neighborhood $O$ of $(x,t)$ for which \eqref{IterAssump}
holds for each $(y,s)\in O$ and $r<\rho_1$. We can then repeat the same computation above to conclude \eqref{DecayofE}. 
\end{proof}

\begin{proof}(of Theorem \ref{PartialRegThm})
Let $\Sigma$ denote the set of points $(x,t)$ for which the following limits hold
$$
\begin{cases}
\lim_{r\rightarrow 0^+}(\bv_t)_{Q_r(x,t)}=\bv_t(x,t)\\
\lim_{r\rightarrow 0^+}(D^2\bv)_{Q_r(x,t)}=D^2\bv(x,t)\\
\lim_{r\rightarrow 0^+}E(x,t,r)=0
\end{cases}.
$$
The first two limits each occur on a set of full measure by Lebesgue's differentiation theorem. As for the third limit, recall Poincar\'{e}'s
inequality on a cylinder $Q_r\subset \Omega\times(0,T)$:
$$
\iint_{Q_r}|\bw - (\bw)_{Q_r}|^2dyds \le C\left\{r^4\iint_{Q_r}|\bw_t|^2dyds + r^2\iint_{Q_r}|D\bw|^2dyds\right\}
$$
for $\bw\in H^1(\Omega\times(0,T);\R^m)$. Choosing $\bw=\bv_{x_i}$ and dividing by $r^2$ above gives
$$
E(x,t,r)\le C\left\{ 
\ffint|\bv_t - (\bv_t)_{Q_r}|^2 dyds+ r^2\ffint\;\; \left|D\bv_t\right|^2dyds + \ffint\;\; \left|D^2\bv-(D^2\bv)_{Q_r}\right|^2dyds\right\}.
$$
By the integrability assumption \eqref{DVTbound} and Lebesgue's differentiation theorem, \newline $\lim_{r\rightarrow 0^+}E(x,t,r)=0$ on a set of full Lebesgue measure.

\par It now follows that for each $(x,t)\in \Sigma$ there is an $L>0$ such that \eqref{IterAssump} holds. 
By  \eqref{DecayofE} and a parabolic version of Campanato's criterion \cite{Camp,Prato}, there is a neighborhood  $O$ of $(x,t)$ where $\bv_t, D^2\bv\in C^\alpha(O)$. Hence, the set $S$ of points $(x,t)$
for which there is some neighborhood of $(x,t)$ where $\bv_t, D^2\bv$ are H\"{o}lder continuous has full Lebesgue measure. By definition, this set $S$ is open. 
\end{proof}


\section{Large time asymptotics}\label{LongSEct}
We now consider the large time behavior of solutions starting with the initial and boundary value problem 
\begin{align}\label{mainIBVP}
\begin{cases}
D\psi(\bv_t)=\Div DF(D\bv), &\quad \Omega\times(0,\infty) \\
\hspace{.42in}\bv = \bh, &\quad \partial\Omega\times [0,\infty)  \\ 
\hspace{.42in}\bv =\bg, &\quad \Omega\times\{0\}
\end{cases}
\end{align} 
with a time independent boundary mapping $\bh:\partial\Omega \rightarrow\R^m$. 
Weak solutions of \eqref{mainIBVP} are defined as in Definition \ref{WeakDefn} once we recall that boundary values of $W^{1,p}(\Omega;\R^m)$ mappings are defined using the trace operator. It is also readily checked that the energy identity \eqref{EnergyIdentity2} holds for solutions of \eqref{mainIBVP}. The intuition is that if $\bv(x,t)=\bh(x)$ for $(x,t)\in \partial\Omega\times [0,\infty)$, 
we may differentiate this boundary condition in time to conclude $\bv_t|_{\partial \Omega}=\bzero$ which allows us to repeat the same proof of \eqref{EnergyIdentity2}. Below, we briefly adapt the ideas from section \ref{ITS} to show \eqref{mainIBVP} admits a weak solution.  We will assume $p\in (1,\infty)$ for the remainder of this section. 

\begin{cor}
Suppose $\bg \in W^{1,p}(\Omega;\R^m)$ and $\bh\in L^p(\partial \Omega;\R^m)$. Then there exists a weak solution of \eqref{mainIBVP}. 
\end{cor}
\begin{proof}
Set $\bv^0=\bg$. Once $\bv^1, \dots, \bv^{k}$ have been determined, we find $\bv^{k+1}$ satisfying
$$
\bv^{k+1}_t\in L^p(\Omega\times[k,k+1]; \R^m), \quad \bv^{k+1}\in L^\infty([k,k+1]; W^{1,p}_0(\Omega;\R^m))
$$
 as a weak solution of the initial and boundary value problem
$$
\begin{cases}
D\psi(\bv_t)=\Div DF(D\bv), &\quad \Omega\times(k,k+1) \\
\hspace{.42in}\bv = \bh, &\quad \partial\Omega\times [k,k+1)  \\ 
\hspace{.42in}\bv =\bv^{k}(\cdot, k), &\quad \Omega\times\{k\}
\end{cases}.
$$
The solution $\bv^{k+1}$ can be constructed using a minor modification of the implicit time scheme \eqref{IFT}.  It is straightforward to verify  
$$
\bv(\cdot,t):=\bv^k(\cdot,t), \quad t\in [k-1,k], \quad k\in \N
$$ 
is a weak solution of \eqref{mainIBVP}. We leave the details to the reader.
\end{proof}

\begin{prop}\label{LargeTimeMan}
Assume $\bg \in W^{1,p}(\Omega;\R^m)$ and $\bh\in L^p(\partial \Omega;\R^m)$, and $\bv$ is a weak solution of \eqref{mainIBVP}. Further suppose condition \eqref{minPSI} holds; that is, $\min_{w\in\R^m}\psi(w)=\psi(0)$. Then the limit
$$
\bw(x):=\lim_{t\rightarrow \infty}\bv(x,t)
$$
exists in $W^{1,p}(\Omega;\R^m)$. Moreover, $\bw$ is a minimizer of the functional $W^{1,p}(\Omega;\R^m)\ni\bu \mapsto \int_\Omega F(D{\bf u})dx$ subject to the boundary 
condition $\bu|_{\partial \Omega}=\bh$.  In particular, 
$$
\begin{cases}
-\text{\normalfont{div}}DF(D\bw) = \bzero,&  \quad x\in \Omega\\
\hspace{.9in} \bw=\bh,& \quad x\in \partial \Omega
\end{cases}. 
$$
\end{prop}
\begin{proof}
Let $s_k$ denote a sequence of positive numbers increasing to $+\infty$ and set $\bv^k(x,t):=\bv(x,t+s_k)$.  By the proof of Lemma \ref{EntropyLem}, $t\mapsto \int_\Omega F(D{\bf v}(x,t))dx$ is nonincreasing. Thus, the mappings in the sequence $\{\bv^k\}_{k\in \N}$ are each solutions of \eqref{mainIBVP} with initial conditions $\{\bv^k(x,0)=\bv(x,s_k)\}_{k\in \N}$ which form a bounded sequence in $W^{1,p}(\Omega;\R^m)$.  Also note for each $t\ge 0$
\begin{align}\label{Energykk}
\int^t_0\int_\Omega D\psi(\bv^k_t(x,s))\cdot \bv^k_t(x,s)dxds  + \int_{\Omega}F(D\bv^k(x,t))dx & =\int_{\Omega}F(D\bv^k(x,0))dx \\
&\le  \int_{\Omega}F(D\bg(x))dx. \nonumber
\end{align}
In particular, by the first part of the proof of Theorem \ref{BlowLem}, there is a subsequence $\bv^{k_j}$ converging to some $\bw$ in $C([0,T]; L^p(\Omega; \R^m))$ for 
each $T>0$ and $\bv^{k_j}(\cdot, t)\rightharpoonup \bw(\cdot, t)$ in $W^{1,p}(\Omega;\R^m)$ for each $t \ge 0$.  

\par Corollary \ref{EntropyLem} and \eqref{Coercive} also imply the limit 
$$
L:=\lim_{\tau\rightarrow \infty}\int_\Omega F(D{\bf v}(x,\tau))dx=\lim_{k\rightarrow \infty}\int_\Omega F(D\bv^k(x,t))dx
$$
exists for each $t\ge 0$.  Passing to the limit as $k\rightarrow \infty$ in \eqref{Energykk} gives 
$$
\limsup_{k\rightarrow \infty}\int^t_0\int_\Omega(\psi(\bv^k_t(x,s)) -\psi(0))dxds+L\le \limsup_{k\rightarrow \infty} \int^t_0\int_\Omega D\psi(\bv^k_t(x,s))\cdot \bv^k_t(x,s)dxds+L=L.
$$
As $\psi(\bv^k_t) -\psi(0)$ is nonnegative, and $\psi$ is strictly convex, it follows that $\bv^k_t\rightarrow \bzero$ in $L^p(\Omega\times [0,T]; \R^m)$ and 
$D\psi(\bv^k_t)\rightarrow \bzero$ in $L^q(\Omega\times [0,T]; \R^m)$ for each $T>0$.

\par Let $\bw_0\in W^{1,p}(\Omega;\R^m)$ be the unique minimizer of $\int_\Omega F(D\bu(x))dx$ subject to the boundary condition $\bu|_{\partial \Omega}=\bh$ . For $t\ge 0$,  
\begin{align*}
\int^t_0\int_\Omega F(D\bw_0(x))dxds &\ge \int^t_0 \int_\Omega F(D\bv^k(x,s))dxds \\
& \quad\quad -  \int^t_0\int_\Omega D\psi(\bv^k_t(x,s))\cdot (\bw_0(x)-\bv^k(x,s))dxds.
\end{align*}
Letting $k=k_j$ and sending $j\rightarrow \infty$ gives, 
$$
\int^t_0\int_\Omega F(D\bw_0(x))dxds\ge Lt\ge \int^t_0\int_\Omega F(D\bw(x,s))dxdt\ge \int^t_0\int_\Omega F(D\bw_0(x))dxds
$$
for $t\ge 0$. Hence, 
$$
\int_\Omega F(D\bw_0(x))dx= L= \int_\Omega F(D\bw(x,t))dx
$$
for each $t\ge 0$ and in particular at $t=0$.  Thus $\bw(\cdot,0)=\bw_0$; and by the strict convexity of $F$, $\bv(\cdot, s_{k_j})\rightarrow \bw_0$ in $W^{1,p}(\Omega;\R^m)$. Since the sequence $\{s_k\}_{k\in \N}$ was arbitrary, it must be that 
$\bv(\cdot, t)\rightarrow \bw_0$ as $t\rightarrow \infty$ in $W^{1,p}(\Omega;\R^m)$.
\end{proof}
Let us now discuss a refinement of Proposition \ref{LargeTimeMan} when $F(M)=\frac{1}{p}|M|^p$, 
$\psi(w)=\frac{1}{p}|w|^p$ and $\bh=\bzero$. The associated initial value problem of interest is \eqref{PflowIVP}
\begin{align*}
\begin{cases}
|\bv_t|^{p-2}\bv_t=\Div(|D\bv|^{p-2}D\bv), &\quad \Omega\times(0,\infty) \\
\hspace{.51in}\bv = \bzero, &\quad \partial\Omega\times [0,\infty)  \\ 
\hspace{.51in}\bv =\bg, &\quad \Omega\times\{0\}
\end{cases}.
\end{align*} 
The previous theorem implies $\lim_{t\rightarrow \infty}\bv(x,t)=\bzero$ in $W^{1,p}_0(\Omega;\R^m)$ for any weak solution $\bv$.  However, 
we will see that more information on the large time behavior of weak solutions is available after exploiting the homogeneity of 
the equation $|\bv_t|^{p-2}\bv_t=\Div(|D\bv|^{p-2}D\bv)$ and the boundary condition $\bv|_{\partial\Omega} = \bzero$.

\par Recall the minimization problem \eqref{Lambdap} involving the optimal $p$-Rayleigh quotient
$$
\Lambda_p:=\inf\left\{\frac{\int_\Omega|D\bu(x)|^pdx}{\int_\Omega|\bu(x)|^pdx}:\bu \in W^{1,p}_0(\Omega;\R^m)\setminus\{\bzero\}\right\}.
$$
Direct methods of calculus of variations imply there is a minimizer $\bu\neq \bzero \in W^{1,p}_0(\Omega;\R^m)$ for $\Lambda_p$, which is also known as a {\it $p$-ground state}. It is also readily verified that $\bu$ is a minimizer if and only if it is a weak solution of the Euler-Lagrange system
\begin{equation}\label{EL}
\begin{cases}
-\Div(|D\bu|^{p-2}D\bu)=\Lambda_p |\bu|^{p-2}\bu, \quad  &x\in \Omega\\
\hspace{1.18in} \bu = \bzero, \quad &x\in \partial \Omega
\end{cases}.
\end{equation}
Thus $\Lambda_p$ has a natural interpretation as a nonlinear eigenvalue.

\par Observe that for any admissible $\bu$ for minimization problem associated with $\Lambda_p$ and orthogonal $m\times m$ matrix $O$, $O\bu$ is admissible and has the same $p$-Rayleigh quotient as $\bu$. Consequently, we cannot expect minimizers to be unique except modulo multiplication by a nonzero scalar and matrix multiplication by an orthogonal matrix.  When $m=1$,  the uniqueness modulo multiplication by scalars was first established using a maximum principle argument by Sakaguchi \cite{Saka}.  Belloni and Kawohl issued a much simpler proof using convexity \cite{BellKaw};  it would be interesting to deduce whether or not their ideas generalize to this vectorial $(m>1)$ setting.  Before studying solutions of \eqref{PflowIVP}, we make a few observations about $p$-ground states and their $p$-Rayleigh quotients. 
\begin{prop}
Let $\lambda_p$ denote the optimal p-Rayleigh quotient when $m=1$. Then 
\begin{equation}\label{LamEst}
m^{-\left|\frac{p}{2}-1\right|}\lambda_p\le \Lambda_p\le \lambda_p.
\end{equation}
In particular, $\Lambda_2=\lambda_2$. 
\end{prop}
\begin{proof}
Choose $u\in W^{1,p}_0(\Omega)\setminus\{0\}$ satisfying 
$$
\lambda_p=\frac{\int_\Omega |Du(x)|^pdx}{\int_\Omega |u(x)|^pdx}
$$
and let $z\in \R^m$ be nonzero vector. Note $\bu(x):=u(x) z\in W^{1,p}_0(\Omega;\R^m)$, $\bu\neq \bzero$ and thus  
$$
\Lambda_p\le \frac{\int_\Omega|D\bu(x)|^pdx}{\int_\Omega|\bu(x)|^pdx}=\frac{|z|^p\int_\Omega |Du(x)|^pdx}{|z|^p\int_\Omega |u(x)|^pdx}=\lambda_p.
$$
\par Now assume $p\ge 2$ and $\bu=(u^1, \dots,u^m)\in W^{1,p}_0(\Omega;\R^m)$. The elementary inequality 
$$
\left(\sum^m_{i=1}|z^i|^p\right)^{1/p}\le \left(\sum^m_{i=1}|z^i|^2\right)^{1/2} \le m^{\frac{1}{2}-\frac{1}{p}} \left(\sum^m_{i=1}|z^i|^p\right)^{1/p}, \quad (z^1,\dots, z^m)\in \R^m
$$
gives
\begin{align*}
\int_\Omega|D\bu(x)|^pdx & =\int_\Omega\left(\sum^{m}_{i=1}|Du^i(x)|^2\right)^{p/2}dx \\
& \ge\int_\Omega \sum^{m}_{i=1}|Du^i(x)|^p dx\\
& \ge \lambda_p \int_\Omega\sum^{m}_{i=1}|u^i(x)|^p dx\\
& \ge \lambda_p m^{-\left(\frac{p}{2}-1\right)} \int_\Omega|\bu(x)|^p dx.
\end{align*}
We are then able to conclude the lower bound in \eqref{LamEst} for $p\ge 2$. A proof for $1\le p\le 2$ can be made similarly. 
 \end{proof}
 \begin{cor}
 Assume $p=2$ and $m>1$. Any $2$-ground state is necessarily of the form 
 \begin{equation}\label{SpecForm}
 \bu(x)=u(x)z, \quad x\in \Omega
 \end{equation}
 where $z\in \R^m$ is a nonzero vector and $u$ is a $2$-ground state for $m=1$.  
 \end{cor}
 \begin{proof}
 Observe that when $p=2$, the system \eqref{EL} decouples: for each $ i=1,\dots, m$
  $$
  \begin{cases}
  -\Delta u^i=\Lambda_2 u^i, \quad & x\in \Omega\\
  \hspace{.25in} u^i=0, \quad & x\in \partial \Omega
  \end{cases} 
  $$
The claim follows  as $\Lambda_2=\lambda_2. $
  \end{proof}
  \begin{rem}
  When $p\neq 2$ and $u$ is $p$-ground state  for $m=1$, $\bu$ defined in \eqref{SpecForm} satisfies \eqref{EL} with $\lambda_p$ replacing $\Lambda_p$. 
  \end{rem}

\par Observe that for weak solutions of \eqref{PflowIVP}, the energy identity takes the simple form 
$$
\int^t_0\int_\Omega |\bv_t(x,s)|^pdxds + \int_\Omega\frac{1}{p}|D\bv(x,t)|^pdx=\int_\Omega\frac{1}{p}|D\bg(x)|^pdx
$$
$t\ge 0$. We can also derive a new estimate for solutions by taking the inner product of both sides of the system
\begin{equation}\label{PFlow}
|\bv_t|^{p-2}\bv_t=\Div(|D\bv|^{p-2}D\bv) \nonumber 
\end{equation} 
with $\bv$ and integrating by parts. To this end, it will be useful for us to recall the constant 
$$
\Upsilon_p=\Lambda_p^{\frac{1}{p-1}}.
$$
defined in \eqref{UpsAndDowns}.

\begin{lem}
Let $\bv$ be a weak solution with $\bg\in W^{1,p}_0(\Omega;\R^m)$. Then for almost every $t>0$ 
\begin{equation}\label{Dvvtbound}
\Upsilon_p \int_\Omega |D\bv(x,t)|^pdx \le \int_\Omega |\bv_t(x,t)|^pdx
\end{equation}
and 
\begin{equation}\label{ExpDecay}
\int_{\Omega}|D\bv(x,t)|^pdx\le e^{-(p\Upsilon_p)t}\int_{\Omega}|D\bg(x)|^pdx
\end{equation}
for each $t>0$.
\end{lem}

\begin{proof} Using $\bv(\cdot, t)$ as a test function in \eqref{WeakSolnCond},
\begin{align}\label{HolderDerBound}
\int_{\Omega}|D\bv(x,t)|^pdx&=\int_{\Omega}|D\bv(x,t)|^{p-2}D\bv(x,t)\cdot D\bv(x,t)dx \nonumber \\
&=-\int_{\Omega}|\bv_t(x,t)|^{p-2}\bv_t(x,t)\cdot \bv(x,t) dx \nonumber \\
&\le \left(\int_{\Omega}|\bv_t(x,t)|^pdx\right)^{1-1/p}\left(\int_{\Omega}|\bv(x,t)|^pdx\right)^{1/p}  \\
&\le \Lambda_p^{-1/p}\left(\int_{\Omega}|\bv_t(x,t)|^pdx\right)^{1-1/p}\left(\int_{\Omega}|D\bv(x,t)|^pdx\right)^{1/p}.\nonumber
\end{align}
This proves \eqref{Dvvtbound}.  Employing \eqref{EnergyIdentity} and \eqref{Dvvtbound} gives
\begin{equation}\label{DiffIneqmu}
\frac{d}{dt}\int_{\Omega}|D\bv(x,t)|^pdx = -p\int_\Omega|\bv_t(x,t)|^pdx\le - p \Upsilon_p \int_{\Omega}|D\bv(x,t)|^pdx.
\end{equation}
Inequality \eqref{ExpDecay} now follows from Gr\"{o}nwall's inequality. 
\end{proof}

\begin{cor}\label{GradUNon} For any weak solution $\bv$ of \eqref{mainIVP},
$$
t\mapsto e^{(\Upsilon_p p)t}\int_{\Omega}|D\bv(x,t)|^pdx
$$
is nonincreasing and 
\begin{equation}\label{ScaledEnergyInequality}
\int^t_0\int_\Omega e^{(\Upsilon_p p) s}(|\bv_t(x,s)|^p - \Upsilon_p |D\bv(x,s)|^p)dxds + e^{(\Upsilon_p p) t}\int_\Omega \frac{|D\bv(x,t)|^p}{p}dx = \int_\Omega \frac{|D\bg(x)|^p}{p}dx
\end{equation}
for $t\ge 0.$ 
\end{cor}
\begin{proof}
The first assertion follows directly from \eqref{DiffIneqmu}, and the second follows from the fundamental theorem of calculus for the Lebesgue integral. 
\end{proof}
\begin{prop}\label{RayleighGoDown} Assume that $\bv$ is a weak solution of \eqref{PflowIVP} such that $\bv(\cdot, t)\neq \bzero$ for each 
$t\ge 0$. Then the $p$-Rayleigh quotient 
$$
[0,\infty)\ni t\mapsto \frac{\int_{\Omega}|D\bv(x,t)|^pdx }{\int_{\Omega}|\bv(x,t)|^pdx }
$$
is nonincreasing.
\end{prop}
\begin{proof}
Making use of \eqref{NaturalSpace}, 
$$
\frac{d}{dt}\int_\Omega\frac{1}{p}|\bv(x,t)|^pdx=\int_\Omega|\bv(x,t)|^{p-2}\bv(x,t)\cdot \bv_t(x,t)dx
$$
holds for Lebesgue almost every time $t>0$; this can be proved by a smoothing argument like the one given in Theorem 3 of section 5.9 of \cite{Evans}. We calculate for almost every $t>0$
\begin{align}\label{RayleighComp}
\frac{d}{dt}\frac{\int_{\Omega}|D\bv|^pdx }{\int_{\Omega}|\bv|^pdx } & = 
-p\frac{\int_{\Omega}|\bv_t|^pdx }{\int_{\Omega}|\bv|^pdx } - p \frac{\int_{\Omega}|D\bv|^pdx }{\left(\int_{\Omega}|\bv|^pdx\right)^2} \int_{\Omega}|\bv|^{p-2} \bv\cdot \bv_tdx \nonumber \\
&=\frac{p}{\left(\int_{\Omega}|\bv|^pdx\right)^2}\left\{\int_{\Omega}|D\bv|^pdx\int_{\Omega}|\bv|^{p-2} \bv\cdot(-\bv_t)dx -\int_{\Omega}|\bv|^pdx \int_{\Omega}|\bv_t|^pdx \right\}.
\end{align}
By H\"{o}lder's inequality 
$$
\int_{\Omega}|\bv|^{p-2} \bv\cdot(-\bv_t)dx\le \left(\int_{\Omega}|\bv|^{p}dx\right)^{1-1/p} \left(\int_{\Omega}|\bv_t|^{p}dx\right)^{1/p},
$$
and combining with \eqref{HolderDerBound} gives
$$
\int_{\Omega}|D\bv|^pdx\int_{\Omega}|\bv|^{p-2} \bv\cdot(-\bv_t)dx \le \int_{\Omega}|\bv|^pdx \int_{\Omega}|\bv_t|^pdx.
$$ 
From \eqref{RayleighComp}, we conclude
$$
\frac{d}{dt} \frac{\int_{\Omega}|D\bv|^pdx }{\int_{\Omega}|\bv|^pdx }\le 0.
$$
\end{proof}

Proposition \ref{RayleighGoDown} exhibits a connection between the doubly nonlinear evolution \eqref{PflowIVP} and the minimization problem \eqref{Lambdap}.  We also remark that if 
$\bg$ is a $p$-ground state, then 
\begin{equation}\label{SepVarSoln}
\bv(x,t)=e^{-\Upsilon_p t}\bg(x)
\end{equation}
is a weak solution of \eqref{PflowIVP}.  In fact, we can use the nonincreasing property of the $p$-Rayleigh quotient of weak solutions to conclude the converse. 
\begin{prop}
Assume that $\bg$ is a $p$-ground state. Then $\bv$ defined by \eqref{SepVarSoln} is the unique weak solution of \eqref{PflowIVP}. 
\end{prop}
\begin{proof}
Suppose $\bv$ is a weak solution of \eqref{PflowIVP}. If $\bv(\cdot, t)\neq \bzero$ for each $t>0$, the monotonicity of the $p$-Rayleigh quotient gives 
$$
\frac{\int_{\Omega}|D\bv(x,t)|^pdx }{\int_{\Omega}|\bv(x,t)|^pdx }\le \frac{\int_{\Omega}|D\bg(x)|^pdx }{\int_{\Omega}|\bg(x)|^pdx }=\Lambda_p.
$$
Therefore, $\bv$ itself is a $p$ ground state and satisfies \eqref{EL}. Since $\bv$ also satisfies \eqref{PflowIVP}, 
\begin{equation}\label{vODE}
\bv_t = -\Upsilon_p \bv
\end{equation}
for almost every $t\ge 0$. As $t\mapsto\bv(\cdot,t)\in L^p(\Omega;\R^m)$ is locally absolutely continuous, \eqref{SepVarSoln} holds for each $t\ge 0$ by integration. 

\par Now suppose that $t=T$ is the first time that $\bv(\cdot, T)=\bzero$.  Observe \eqref{vODE} holds for almost every 
$t\in (0,T)$, and integrating this equation from $t=0$ to $t=T$ gives that \eqref{SepVarSoln} holds for $t=T$. However, this contradicts 
the definition of $T$. 
\end{proof}

\begin{proof} (of Theorem \ref{HomogenousConvergence})
Let $\{s_k\}_{k\in \N}$ be an increasing sequence of positive numbers tending to $\infty$, $T>0$ and set $\bv^k(x,t):=e^{\Upsilon_p s_k}\bv(x,t+s_k)$. Notice
 that $\bv^k$ is a weak solution of \eqref{PflowIVP} and $\{\bv^k(\cdot, 0)\}_{k\in \N}$ is bounded in $W^{1,p}_0(\Omega;\R^m)$.  By Theorem \ref{BlowLem}, there is 
 a subsequence $\bv^{k_j}$ tending to a weak solution $\bw$ of \eqref{PflowIVP}  in $C([0,T],L^p(\Omega; \R^m))$ and in $L^p([0,T]; W^{1,p}_0(\Omega;\R^m))$. 
Without loss of generality, we assume $\bw(\cdot, t)=\lim_{j\rightarrow\infty}\bv^{k_j}(\cdot,t)$ in $W^{1,p}_0(\Omega;\R^m)$ for almost every $t\in [0,T]$ since this convergence takes place for a subsequence. 

\par Define $\bu(x,t)=e^{\Upsilon_p t}\bw(x,t)$ and note
\begin{align*}
S:&=\lim_{\tau\rightarrow \infty}e^{p\Upsilon_p \tau}\int_\Omega |D\bv(x,\tau)|^pdx \nonumber \\
& =\lim_{j\rightarrow \infty}e^{p\Upsilon_p (t+s_k)}\int_\Omega |D\bv(x,t+s_k)|^pdx \nonumber \\
& =\lim_{j\rightarrow \infty}e^{p\Upsilon_pt}\int_\Omega |D\bv^k(x,t)|^pdx \nonumber \\
&=\int_\Omega |D\bu(x,t)|^pdx
\end{align*}
for almost every $t\in [0,T]$. However, as $\bw$ is a weak solution of \eqref{PflowIVP}, $t\mapsto \int_\Omega |D\bw(x,t)|^pdx$ is absolutely continuous and
thus $S=\int_\Omega |D\bu(x,t)|^pdx$ holds for each $t\in [0,T]$.   Moreover, as in the proof of \eqref{ScaledEnergyInequality}, 
\begin{align}\label{LambdaLimit}
0&=\frac{d}{dt}\int_\Omega |D\bu(x,t)|^pdx \nonumber \\
&=-\int_\Omega |\bu_t(x,t)-\Upsilon_p \bu(x,t)|^{p-2}(\bu_t(x,t)-\Upsilon_p \bu(x,t))\cdot \bu_t(x,t)dx \\
&=- \int_\Omega |\bu_t(x,t)-\Upsilon_p \bu(x,t)|^{p}dx +\Upsilon_p \int_\Omega |D\bu(x,t)|^pdx\nonumber
\end{align}
for almost every $t \ge 0$. It now follows from the convexity of $z\mapsto \frac{1}{p}|z|^p$ and \eqref{LambdaLimit} that
\begin{align*}
\Upsilon_p \int_\Omega |D\bu(x,t)|^pdx & =  \int_\Omega |\bu_t(x,t)-\Upsilon_p \bu(x,t)|^{p}dx\\
& \le  \int_\Omega | \Upsilon_p\bu(x,t)|^{p}dx \\
&\quad + p\int_\Omega |\bu_t(x,t)-\Upsilon_p \bu(x,t)|^{p-2}(\bu_t(x,t)-\Upsilon_p \bu(x,t))\cdot \bu_t(x,t)dx \\
& =  \Upsilon_p ^p\int_\Omega |\bu(x,t)|^{p}dx.
\end{align*}
\par Since $S\neq 0$, $\bu(\cdot, t)\neq \bzero$ and thus 
$$
\Lambda_p=\frac{\int_\Omega |D\bu(x,t)|^pdx}{\int_\Omega |\bu(x,t)|^{p}dx}=\frac{\int_\Omega |D\bw(x,t)|^pdx}{\int_\Omega |\bw(x,t)|^{p}dx}
$$
for each $t\ge 0$.  So for any $t_0\in [0,T]$ such that $\bv^{k_j}(\cdot, t_0)\rightarrow \bw(\cdot,t_0)$ in $W^{1,p}_0(\Omega;\R^m)$, 
\begin{align*}
\lim_{t\rightarrow\infty} \frac{\int_{\Omega}|D\bv(x,t)|^pdx }{\int_{\Omega}|\bv(x,t)|^pdx }&=\lim_{j\rightarrow\infty} \frac{\int_{\Omega}|D\bv(x, t_0+s_{k_j})|^pdx }{\int_{\Omega}|\bv(x,t_0+s_{k_j})|^pdx }\\
&=\lim_{k\rightarrow\infty} \frac{\int_{\Omega}|D\bv^{k_j}(x,t_0)|^pdx }{\int_{\Omega}|\bv^{k_j}(x,t_0)|^pdx }\\
&=\frac{\int_\Omega |D\bw(x,t_0)|^pdx}{\int_\Omega |\bw(x,t_0)|^{p}dx}\\
&=\Lambda_p.
\end{align*}
This proves \eqref{LamLimit}. 

\par For an increasing sequence of times $t_k\rightarrow \infty$, suppose $\{e^{\Upsilon_p t_k}\bv(\cdot, t_k)\}_{k\in \N}$ converges to some  $\bw$ weakly in $W_0^{1,p}(\Omega; \R^m)$.  By Rellich compactness, there is a further subsequence $\{e^{\Upsilon_p t_{k_j}}\bv(\cdot, t_{k_j})\}_{j\in \N}$ that converges in $L^p(\Omega;\R^m)$. As $S>0$, $\bw\neq \bzero$ and so 
$$
\Lambda_p=\lim_{j\rightarrow \infty} \frac{\int_{\Omega}|D\bv(x, t_{k_j})|^pdx }{\int_{\Omega}|\bv(x,t_{k_j})|^pdx }
=\lim_{j\rightarrow \infty} \frac{\int_{\Omega}|D(e^{\Upsilon_p t_{k_j}}\bv(x, t_{k_j}))|^pdx }{\int_{\Omega}|e^{\Upsilon_p t_{k_j}}\bv(x,t_{k_j})|^pdx }\ge\frac{\int_\Omega |D\bw(x)|^pdx}{\int_\Omega |\bw(x)|^{p}dx}\ge \Lambda_p.
$$
Therefore, $\bw$ is a $p$-ground state and 
$$
\int_{\Omega}|D(e^{\Upsilon_p t_{k_j}}\bv(x, t_{k_j}))|^pdx \ge \int_\Omega |D\bw(x)|^pdx.
$$
As a result, $\{e^{\Upsilon_p t_{k_j}}\bv(\cdot, t_{k_j})\}_{j\in \N}$ converges to $\bw$ strongly in $W_0^{1,p}(\Omega; \R^m)$. 
\end{proof}
We conjecture that the full limit  $\lim_{t\rightarrow \infty}e^{\Upsilon_pt}\bv(x,t)$ exists in $W^{1,p}_0(\Omega;\R^m)$ and is a $p$-ground state (when it does not vanish identically) for each $m\ge 1$. This is the case when $p=2$
as the system \eqref{PflowIVP} decouples into $m$ separate heat equations.  It may also be the case that the full limit holds for solutions obtained by way of the implicit scheme as additional bounds are available. We conclude this discussion with such an estimate.

\begin{prop}
Let $\bv$ denote a weak solution of \eqref{PflowIVP} as described in Proposition \ref{ExistProp}. Then 
$$
t\mapsto \int_{\Omega}|D\bv(x,t)|^pdx
$$ 
is convex and for $h>0$,
$$
\int_{\Omega}|\bv_t(x,t)|^pdx\le \frac{e^{-(p\Upsilon_p)(t-h)}}{h}\int_\Omega\frac{1}{p}|D\bg(x)|^pdx
$$
for almost every $t\ge h$. 
\end{prop}

\begin{proof}
From Proposition \ref{BlowLem}, 
$$
\int_\Omega|\bv_t(x,t_2)|^pdx\le\int_\Omega|\bv_t(x,t_1)|^pdx
$$
for almost every $(t_1,t_2)\in [0,\infty)\times [0,\infty)$, $t_1\le t_2$.  The function $t\mapsto \int_{\Omega}|D\bv(x,t)|^pdx$ is then convex as 
$\frac{d}{dt}\int_{\Omega}|D\bv(x,t)|^pdx=-p\int_\Omega|\bv_t(x,t)|^pdx$ for almost every $t\ge 0$.  

\par  By \eqref{ExpDecay} and convexity, for almost every $t\ge h$
\begin{align*}
 e^{-(p\Upsilon_p)(t-h)}\int_{\Omega}\frac{1}{p}|D\bg(x)|^pdx & \ge \int_{\Omega}\frac{1}{p}|D\bv(x,t-h)|^pdx \\
& \ge  \int_{\Omega}\frac{1}{p}|D\bv(x,t)|^pdx + \left(-\int_\Omega|\bv_t(x,t)|^pdx\right)((t-h)-t)\\
&\ge h\int_\Omega|\bv_t(x,t)|^pdx.
\end{align*}

\end{proof}


\end{document}